\documentclass[12pt,pdftex]{amsart}
\usepackage[margin=1in]{geometry}	

\usepackage{amsmath,amsthm,amssymb,amsfonts,bm, multicol}
\usepackage[mathscr]{euscript}
\usepackage{mathtools}
\usepackage{url, hyperref}
\AtBeginDocument{}

\usepackage{tikz}
\usepackage{tkz-graph}
\usetikzlibrary{matrix,arrows,shapes}

\theoremstyle{plain}
\newtheorem{thm}{Theorem}
\newtheorem{lemma}[thm]{Lemma}
\newtheorem{prop}[thm]{Proposition}
\newtheorem{cor}[thm]{Corollary}

\newtheorem{case*}{Case}

\numberwithin{subcase}{case}

\theoremstyle{definition}
\newtheorem{defn}[thm]{Definition}
\newtheorem{ex}[thm]{Example}

\theoremstyle{remark}

\newcommand{\sT}{\mathscr{T}}
\newcommand{\sA}{\mathscr{A}}
\newcommand{\sB}{\mathscr{B}}

\newcommand{\zz}{\mathbb{Z}}

\newcommand{\xx}{\bm{x}}

\newcommand{\sg}{\mathscr{S}}

\newcommand{\rL}{\widetilde{L}}
\newcommand{\rV}{\widetilde{V}}

\newcommand{\smax}{\mathrm{max}}

\DeclareMathOperator{\weight}{wt}

\graphicspath{{./figures/}}

\begin{document}

\title{The Avalanche Polynomial of a Graph}

\author[D. Austin]{Demara Austin}
\address[D. Austin]{Department of Mathematics, Kansas State
  University, Manhattan, KS 66506}%{St. Mary's College of Maryland}
\email{\href{mailto:draustin@math.ksu.edu}{draustin@math.ksu.edu}}
\author[M. Chambers]{Megan Chambers}
\address[M. Chambers]{Department of Mathematics, North Carolina State
  University, Raleigh, NC 27695-8205}%{Youngstown State University}
\email{\href{mailto:mjchambe@ncsu.edu}{mjchambe@ncsu.edu}}
\author[R. Funke]{Rebecca Funke}
\address[R. Funke]{Department of Math \& Computer Science,
University of Richmond, VA 23173}
\email{\href{mailto:becca.funke@gmail.com}{becca.funke@gmail.com}}
\author[L. D. Garc\'{\i}a Puente]{Luis David Garc\'{\i}a Puente}
\address[L. D. Garc\'{\i}a Puente]{Department of Mathematics and
  Statistics, Sam Houston State University, Huntsville, TX 77341-2206}
\email{\href{mailto:lgarcia@shsu.edu}{lgarcia@shsu.edu}}
\urladdr{\url{http://www.shsu.edu/ldg005}}
\author[L. Keough]{Lauren Keough}
\address[L. Keough]{Department of Mathematics,
  Grand Valley State University, Allendale, MI 49401}
\email[Corresponding author]{\href{mailto:keough.lauren@gmail.com}{keough.lauren@gmail.com}}
%\email{s-lkeough1@math.unl.edu}
%\urladdr{http://www.math.unl.edu/$\sim$s-lkeough1}
%\urladdr{\url{http://www.math.unl.edu/~s-lkeough1}}

\begin{abstract}
The (univariate) avalanche polynomial of a graph, introduced by Cori,
Dartois and Rossin in 2006, captures the distribution of the length of
(principal) avalanches in the abelian sandpile model. This polynomial
has been used to show that the avalanche distribution in the sandpile
model on a multiple wheel graph does not follow the expected power law
function. In this article, we introduce the (multivariate) avalanche
polynomial that enumerates the toppling sequences of all principal
avalanches. This polynomial generalizes the univariate avalanche
polynomial and encodes more information. In particular, the avalanche
polynomial of a tree uniquely identifies the underlying tree. In this
paper, the avalanche polynomial is characterized for  trees, cycles,
wheels, and complete graphs.  
%\todo[inline]{The abstract should be ``about 150 words" for Involve.
%This is 109.} 
\end{abstract}



\thanks{This research was supported by NSF grants DMS-1045082, DMS-1045147 , and NSA grant
  H98230-14-1-0131. The fourth author was also partially supported by Simons
  Foundation Collaboration grant 282241.}
\thanks{The authors would like to thank the 2014 PURE Math program and
the University of Hawai'i at Hilo for making this collaboration
possible.}
% and providing terrific accomodations and creating a
%estimulating and productive research environment.}

\maketitle

\section{Introduction}

The area of \emph{sandpile groups} is a flourishing area that started
in mathematical physics in 1987 with the seminal work of Bak, Tang and
Wiesenfeld \cite{BTW88}. Since then it has found many, often unexpected,
applications in diverse areas of mathematics, physics, computer
science and even some applications in the biological sciences and
economics.  In thermodynamics, a \emph{critical point} is the end point of a
phase equilibrium curve. The most prominent example is the
liquid-vapor critical point, the end point of the pressure-temperature
curve at which the distinction between liquid and gas can no longer be
made. In order to drive this system to its critical point it is
necessary to tune certain parameters, namely pressure and temperature.  

In nature, one can also observe different types of dynamical systems
that have a critical point as an attractor. The macroscopic behavior
of these systems displays the spatial and/or temporal scale-invariance
characteristic of the critical point of a phase transition, but
without the need to tune control parameters to precise values. Such a
system is said to display \emph{self-organized criticality}, a concept first
introduced in 1987 by  Bak, Tang and 
Wiesenfeld in their groundbreaking paper. This concept
is thought to be present in a large variety of physical systems like
earthquakes  \cite{SS89, CBO91}, forest fires and in stock
market fluctuations \cite{B96}. Self-organized criticality is
considered to be one of the mechanisms by which \emph{complexity}
arises in nature \cite{BP95} and has been extensively studied 
in the statistical physics literature during the last three decades
\cite{J98, P12, MG14}.   

In \cite{BTW88}, Bak, Tang and Wiesenfeld conceived a \emph{cellular
  automaton} model as a paradigm of self-organized criticality.  This
model is defined on a rectangular grid of cells as shown in 
Figure \ref{grid}. The system evolves in discrete time such that at
each time step a sand grain is 
dropped onto a random grid cell.  When a  cell
amasses four grains of sand, it becomes \emph{unstable}.  It relaxes by
\emph{toppling}  whereby four sand grains leave the site, and each of the
four neighboring sites gets one grain. If the unstable cell is on the
boundary of the grid  then, depending on whether the cell is a corner
or not, either
one or two sand grains fall off the edge and disappear. As the sand
percolates over the grid in this fashion, adjacent cells may
accumulate four grains of sand and become unstable causing an
\emph{avalanche}. This settling process continues until all cells are
\emph{stable}. Then another cell is picked randomly, the height of  the sand
on that grid cell is increased by one, and the process is repeated. 

\begin{figure}[!htb]
\begin{center}
\includegraphics[scale=.8]{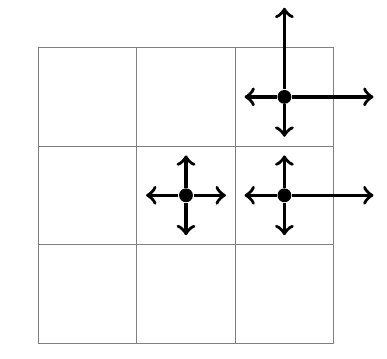}
\end{center}
\caption{The Bak, Tang and Wiesenfeld model of self-organized criticality.}
\label{grid}
\end{figure}

Imagine starting this process on an empty grid. 
At first there is little
activity, but as time goes on, the size (the total number of topplings
performed) and extent of the avalanche
caused by a single grain of sand becomes hard to predict. Figure
\ref{soc} shows the distribution of avalanches in a computational
experiment performed on a $20\times 20$ grid. Starting with the 
\emph{maximal stable sandpile}, i.e., the sandpile with three grains
of sand at each site, a total of $100,000$ sand grains were added at random,
allowing the sandpile to stabilize in between. 
Many authors have studied the distribution of the sizes of the
avalanches for this model showing that it obeys a power law with
exponent $-1$ \cite{C91, DRSV95, KLGP00}.

\begin{figure}[!htb]
\begin{center}
\includegraphics[scale=.4]{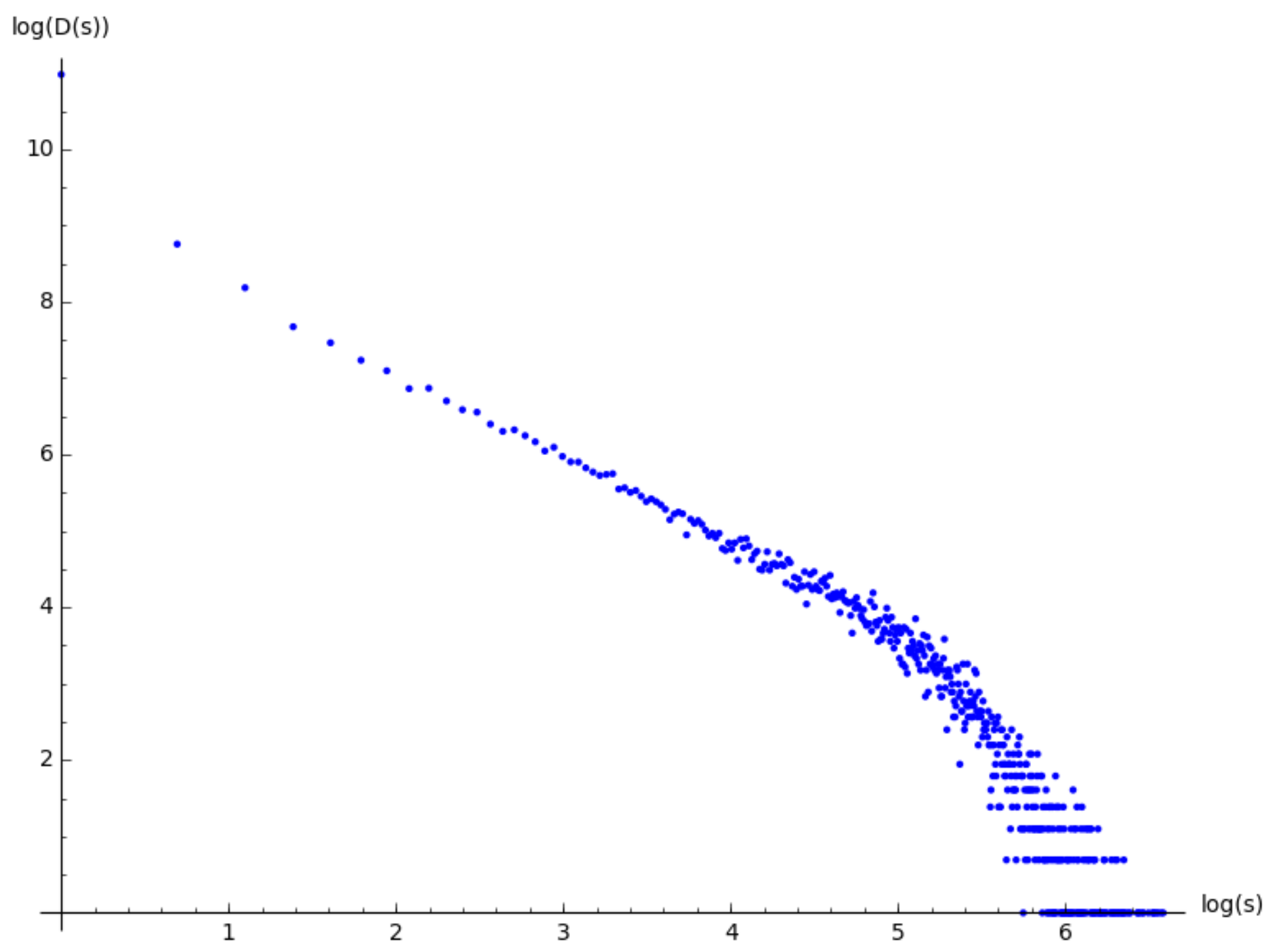}
\end{center}
\caption{Distribution of avalanches $D(s)$ as
  a function of the size $s$.}
 %on a $20\times 20$ rectangular grid.}
\label{soc}
\end{figure}

In 1990, Dhar generalized the Bak, Tang and Wiesenfeld  model
replacing the rectangular grid with an arbitrary combinatorial graph
\cite{D90}. In this model,  known as the \emph{(abelian) sandpile model},
the sand grains are placed  
at the vertices of the graph and the toppling threshold depends on the
degree (outdegree) of each vertex. Certain conditions on the graph
(the existence of a \emph{global sink vertex}) ensure that any
avalanche terminates after a finite number of topplings.  
The sandpile model was also considered by combinatorialists as a game on a
graph called the \emph{chip firing game} \cite{BLS91,
  BL92, B99}. 

The long-term behavior
of the abelian sandpile model on a graph is encoded by the \emph{critical
configurations}, also known as \emph{recurrent sandpiles}. These
critical configurations have
connections to  \emph{parking functions} \cite{BCT10}, to the
\emph{Tutte polynomial} \cite{Merino05}, and to the lattices of \emph{flows}
and \emph{cuts} of a graph \cite{BDN97}.  Among other properties, the
critical configurations have the structure of a
\emph{finite abelian group}. This group has been discovered in several
different contexts and received many 
names: the {\it sandpile group} \cite{D90, HLMPPW08}, the {\it
  critical group}~\cite{B99}, the {\it group of 
bicycles} \cite{B86}, the {\it group of components}~\cite{L89}, and the
{\it jacobian} of a graph~\cite{BN07}.

The fact that the distribution of avalanches on the grid follows a
power law has been a focal point from the statistical physics
perspective. A natural question is what type of distributions do we
get in the more general setting introduced by Dhar. In this paper, we
focus on this question. In fact, we go one step beyond just finding such
distributions. We actually describe the combinatorial structure of each
avalanche for certain families of graphs.

Experiments on the distribution of sizes of the avalanches have been
mostly restricted to the cases of rectangular grids and some classes
of regular graphs. However, very
little is known for arbitrary graphs \cite{DaRoFPSAC03}. 
In 2004, Cori, Dartois, and Rossin 
introduced the (univariate)
\emph{avalanche polynomial} that encodes the sizes of  \emph{principal
avalanches}, that is, avalanches resulting from adding a single grain
of sand to a recurrent sandpile \cite{CDR04}. These
authors completely describe the avalanche polynomials for 
trees, cycles, complete and lollypop graphs. Moreover, using these
polynomials, they show that the resulting sandpile models on
these graphs no longer obey the power law observed in the rectangular
grid. 
In 2003, Dartois and Rossin obtained exact results for the avalanche
distribution on wheel graphs
%using techniques from \emph{regular languages} and \emph{transducers} 
\cite{DaRoFPSAC03}. 
In 2009, Cori, Micheli and Rossin studied further properties of
the avalanche polynomial on plane trees \cite{CMR09}. In particular,
they show that the avalanche polynomial of a tree does not uniquely
characterize the tree. They also give closed formulas for the average
and variance of the avalanche distribution on trees. 
 
In this paper we introduce the \emph{(multivariate) avalanche
  polynomial}, i.e., a  multivariate polynomial encoding the toppling
sequences of all principal avalanches. This polynomial generalizes the
univariate avalanche polynomial and encodes more information. 
%In particular, we
%show that the multivariate avalanche polynomial of a tree uniquely
%characterizes the underlying tree. 
In  Section \ref{background}, we
describe the sandpile model on an undirected graph and
introduce both the univariate and multivariate avalanche
polynomials. We also present some particular evaluations of the latter
polynomial. In particular, one such evaluation gives rise to the
unnormalized distribution of burst sizes, that is, the number of
grains of sand that fall into the sink in a principal avalanche. In
Section \ref{trees}, we characterize the avalanche 
polynomial of a tree. We also prove in Theorem \ref{uniqueness} that
this polynomial uniquely characterizes its underlying tree. In
Section \ref{cycles}, we compute the multivariate avalanche 
polynomial for cycle graphs and in Section \ref{complete}, we compute
this polynomial for complete graphs using the bijection among
recurrent sandpiles and parking functions. Our
arguments fix some details in the proof of Proposition 5 in
\cite{CDR04} that enumerates the number of principal avalanches of
positive sizes in the complete graph. In Section \ref{wheel}, we
compute the avalanche polynomial for wheel graphs. Our methods 
simplify the arguments in \cite{DaRoFPSAC03} where the authors use
techniques from \emph{regular languages}, \emph{automatas} and \emph{transducers} to characterize the recurrent
sandpiles and determine the exact distribution of avalanche lengths in the wheel graph.  
% Finally, in Section \ref{final}
% we provide some conclusions and point to future work. 

%%%%%%%%%%%%%%%%%%%%%%%%%%%%%%%%%%%%%%%%%%%%%%%%%%%%%%%%%%%%%%%%%%%%

\section{The Abelian Sandpile Model}\label{background}
The \emph{abelian sandpile model} is defined both on directed and
undirected graphs, but here we focus on families of undirected
multigraphs without loops.

\begin{defn} An \emph{(undirected) graph} $G$ is an ordered pair
$(V,E)$, where $V$ is a finite set and $E$ is a finite multiset of the
set of 2-element subsets of $V$.  The elements of $V$ are called
\emph{vertices} and the elements of $E$ are called \emph{edges}.
Given an undirected graph $G = (V,E)$, the \emph{degree} $d_v$ of
a vertex $v \in V$ is the number of edges $e\in E$ with $v \in e$.
For a pair of vertices $u,v\in V$, the weight $\weight(u,v)$ is the
number of edges between $u$ and $v$. We say $u$ and $v$ are adjacent
if $\weight(u,v)>0$.
\end{defn}

Here, we will always assume that our graphs are
\emph{connected}. Moreover, given a graph $G = (V,E)$ we will
distinguish a vertex $s\in V$ and call it a \emph{sink}. The resulting
graph will be denoted $G = (V,E,s)$. 
We will also denote the set of all non-sink vertices by 
$\rV = V\setminus \{s\}$.

\begin{defn} A \emph{sandpile} $c$ on $G = (V,E,s)$ is a
function $c: \rV \to \zz_{\geq 0}$ from the non-sink vertices of
$G$ to the set of non-negative integers, where 
$c(v)$ represents the number of grains of sand at vertex $v$.
We call $v$ \emph{unstable} if $c(v)
\geq d_{v}$. An unstable vertex $v$ can \emph{topple}, resulting in a
new sandpile $c'$ obtained by moving one grain of sand along each of
the $d_{v}$ edges emanating from $v$; that is, $c'(w) = c(w) +
\weight(v,w)$ for all $w\neq v$ and $c'(v) = c(v) - d_{v}$. A sandpile
is \emph{stable} if $c(v) < d_{v}$ for every non-sink vertex $v$ and
\emph{unstable} otherwise.
\end{defn}

The following proposition justifies the name ``abelian sandpile
model'' and it was first proved by Dhar in \cite{D90}.

\begin{prop}\label{prop:THEstabilization} 
Given an unstable sandpile
$c$ on a graph $G = (V, E, s)$, any sequence of topplings of
unstable vertices will lead to the same stable sandpile.
\end{prop} 

%Proposition \ref{prop:THEstabilization} implies that one can talk
%about \emph{the} stabilization of a sandpile.  

%\begin{defn}
Given a sandpile $c$, if $c'$ is obtained from $c$ after a sequence of
sand additions and topplings, we say that $c'$ is \emph{accessible}
form $c$ and we call $c'$ a \emph{successor} of $c$. We denote this by
$c \leadsto c'$. Moreover, if $c'$ is obtained from $c$ by a sequence
of topplings, then the \emph{toppling vector} $f$ associated to the
stabilization $c \leadsto c'$ is the integer vector indexed by the
non-sink vertices of $G$ with $f(v)$ equal to the number of times
vertex $v$ appears in the vertex toppling sequence that sends $c$ to
$c'$. Finally, given a sandpile $c$, the \emph{unique stable} sandpile
obtained after a sequence of topplings is denoted by $c^{\circ}$ and
is called the \emph{stabilization} of $c$. 
%\end{defn}

Let $G = (V, E, s)$ be a graph and $a, b$ be
two sandpiles on $G$. Then $a + b$ denotes the sandpile obtained
by adding the grains of sand vertex-wise, that is, $(a + b)(v) = a(v)
+ b(v)$ for each $v \in \rV$. Note that even if $a$ and $b$
are stable, $a+b$ may not be. We denote the stabilization
of $a + b$ by $a \oplus b$, that is, $a \oplus b = (a +
b)^{\circ}$. The binary operator $\oplus$ is called \emph{stable
addition}.

\begin{defn} A sandpile $c$ is \emph{recurrent} if it is stable
and given any sandpile $a$, there exists a sandpile $b$ such that
$$a\oplus b = c.$$ 
\end{defn}

As an example, given a graph $G = (V, E, s)$, the sandpile $\smax$
defined by $\smax(v) = d_v-1$ for each $v$ in $\rV$ is
recurrent. We call $\smax$ the maximal stable sandpile on $G$.
The following well-known proposition gives a simpler way to compute
the recurrent 
sandpiles. %of a graph $G$.
 
\begin{prop}\label{prop:RecurrentCondition} A stable sandpile $c$
is recurrent if and only if there exists a sandpile $b$ with 
$$\smax \oplus b = c.$$
\end{prop}

As mentioned before, the set of recurrent sandpiles on a graph
$G$ under stable addition forms a finite abelian group denoted
$\sg(G)$, see \cite{DRSV95}. Explicitly, the sandpile group of
$G$ is isomorphic to the cokernel of the  \emph{reduced Laplacian}
matrix of $G$. 
Moreover, this matrix
can be used to compute the stabilization of a sandpile algebraically. 
 
\subsection{Graph Laplacians}

\begin{defn} Let $G$ be a  graph with $n$ vertices
  $v_{1},v_{2},\dots, v_{n}$. The \emph{Laplacian}
of $G$, denoted $L = L(G)$, is the $n\times n$ matrix defined by
\[L_{ij} =
\begin{cases} -\weight(v_{i},v_{j}) & \text{for } i\neq j,\\ d_{v_{i}} &
\text{for } i=j.
\end{cases}
\]
\end{defn}

The \emph{reduced Laplacian} of a graph $G = (V, E, s)$, denoted, 
$\rL = \rL(G)$, is the matrix obtained by deleting the row and column
corresponding to the sink $s$ from the matrix $L$. 
Kirchhoff's  Matrix-Tree Theorem implies that the number of
recurrent sandpiles in $G$, that is, the determinant of reduced
Laplacian $\rL$ equals the number of spanning trees of $G$.
From our definition
we also have that if $c
\leadsto c'$ by toppling vertex $v$, then $c' = c - \rL 1_{v}$, where
$1_{v}$ denotes the (column) vector with
$1_{v}(v) = 1$ and $1_{v}(w) = 0$, for all $w\neq v$ in $\rV$. The
previous observation leads to the following result.  

\begin{prop}\label{prop:laplacian} 
Given a graph $G = (V, E, s)$, and a
sandpile $c$.  If $c \leadsto c'$ by a sequence of
topplings, then $c' = c - \rL f$, where $f$ is the (column)
toppling vector associated to  $c \leadsto c'$. 
%and $f^{t}$ denotes its transpose.
\end{prop}

%Before finding the toppling polynomials for several families of
%graphs, we state and prove a useful lemma.  
The following result, known as Dhar's Burning Criterion, gives an
alternative and useful way to 
characterize recurrent sandpiles in an undirected graph.
Given $G = (V,E,s)$, let $u$ denote the sandpile given by $u_j =
\weight(v_j,s)$ for each $v_j\in \rV$. We will refer to the sandpile
$u$ as the sandpile obtained by `firing the sink'.

 \begin{prop}[{\cite[Corollary 2.6]{CR00}}] \label{prop:Burn} 
The sandpile $c$ is recurrent if and only if $u \oplus c =
c$. Moveover, the firing vector in the stabilization of $u \oplus c$
is $(1,\dots,1)$. 
\end{prop}

%Note that for two vectors $a=(a_1,\dots,a_n)$ and $b=(b_1,\dots,b_n)$
%we say $a\geq b$ if $a_i\geq b_i$ for $1\leq i \leq n$.

%\begin{prop}
%\label{prop:CompareNu} Let $b$, $u$, and $u'$ be stable sandpiles on a
%graph $G$. If $u \geq u'$ then $\nu(b \oplus u) \geq \nu(b \oplus
%u')$.
%\end{prop}

%Now, the useful lemma.

\begin{cor}\label{lem:ToppleOnce} Let $G = (V,E,s)$ be a graph
and $c$ a recurrent sandpile on $G$. 
When one grain of sand is added to a
vertex adjacent to the sink then every vertex can topple at most once.
\end{cor}
%\begin{proof} 
%Let $v$ be some vertex adjacent to the sink. Note $\pi
%\geq 1_v$, so by Propositions \ref{prop:Burn} and \ref{prop:CompareNu}
%$$(1,1,\dots,1) = \nu(\pi \oplus c) \geq \nu(1_v \oplus c).$$
%Therefore, when sand is added to a vertex adjacent to the sink each
%vertex can topple at most once.
%\end{proof}

\subsection{Avalanche Polynomials} The (univariate) avalanche
polynomial was introduced % by Cori, Dartois and Rossin 
in \cite{CDR04}. This
polynomial enumerates the sizes of all principal avalanches.

\begin{defn} Let $G = (V,E,s)$ be a graph and let $v\in
\rV$. 
%Let $1_{v}$ denote the sandpile $1_{v}(v) = 1$ and $1_{v}(w)=0$ for
%all $w\in \rV$ with $w\neq v$. 
Let $c$ be a recurrent sandpile $c$ on 
$G$ and $v\in \rV$, the
\emph{principal avalanche} of $c$ at $v$ is the %firing vector
sequence of vertex topplings
resulting from the stabilization of the sandpile $c+1_v$.
\end{defn}

\begin{defn} The avalanche polynomial for a graph $G$ is defined
as
\[ \sA_{G}(x) = \sum \lambda_{m}x^{m}, \] where $\lambda_{m}$ is the
number of principal avalanches of size $m$.
\end{defn}

%Example \ref{ex:CycleAvalanche} illustrates how to find the avalanche
%polynomial of a 3-cycle.

\begin{ex} \label{ex:CycleAvalanche} 
In this example, we consider the 3-cycle $C_3$. The
recurrent sandpiles on $C_3$ are $(1,0), (0,1),$ and $(1,1)$.  
The table in Figure \ref{table:C3} records
the size of the avalanche for the corresponding recurrent and vertex.
Therefore, %the avalanche polynomial for $C_{3}$ is 
$\sA_{C_{3}} (x) = 2x^{2} + 2x + 2$.

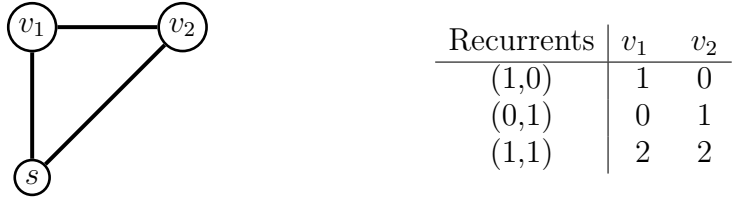
\begin{figure}[!htbp]
\begin{minipage}[c]{0.3\linewidth}
\centering
\begin{tikzpicture} 
\tikzstyle{VertexStyle} = [shape = circle, draw = black,
                                          inner sep = 2pt, outer sep = 0.5pt, minimum size = 2mm,
                                         line width = 1pt]
\SetUpEdge[lw=1.5pt]
\SetGraphUnit{2} 
\tikzset{EdgeStyle/.style={-}}
\Vertex[Math,L=v_1]{1}
\EA[Math,L=v_2](1){2}
\SO[Math,L=s](1){3}
\Edges(1,2,3,1)
\end{tikzpicture}
\end{minipage}
\begin{minipage}[c]{0.45\linewidth}
\centering
\begin{tabular}{c|cc}
{Recurrents} & $v_{1}$ \hspace{2pt} & $v_{2}$ \\ \hline
(1,0) & 1 & 0 \\
(0,1) & 0 & 1 \\
(1,1) & 2 & 2 
\end{tabular}%\hspace{.5in}
\end{minipage}
\caption{Principal avalanche sizes on $C_{3}$.}
\label{table:C3}
\end{figure}

% So for instance, when a grain of sand of sand is added to $v_2$ on the
%recurrent $(1,1)$, the sandpile is then $(1,2)$.  Because $2 \geq
%d_{v_2}$, $v_2$ will topple sending one grain of sand to the sink and
%one to $v_1$. Now the sandpile is $(2,0)$.  Since $2 \geq d_{v_1}$,
%$v_1$ will topple sending one grain of sand to go to the sink and one
%to $v_2$. Now the sandpile is $(1,0)$ which is stable.  So, in the
%stabilization of $(1,1)+1_{v_2}$, two vertices toppled.  In the table
%the row for the recurrent $(1,1)$ under the column denoted with $v_1$,
%the size of the principal avalanche of $(1,1)$ at $v_1$ is recorded as
%2.  Continue in this manner for all three recurrents and both nonsink
%vertices.  The size of these avalanches are encoded in the avalanche
%polynomial. There are two avalanches of size two giving the term
%$2x^2$, two avalanches of size 1 giving the term $2x$, and two
%avalanches of size zero giving the constant term $2$.
\end{ex}

\subsection{Multivariate Avalanche Polynomial}

The (univariate) avalanche polynomial does not contain any information
regarding which vertices topple in a given  principal avalanche. Here
we introduce the \emph{multivatiate avalanche polynomial} that encodes
this information.  
%and hence provides more information about the
%structure of the principal avalanches.

\begin{defn} Let $G = (V,E,s)$ be a graph on $n+1$ vertices
and let $\rV = \{v_1,\dots,v_n\}$. Given $k \in \{1,\dots, n\}$ and a
recurrent sandpile $c$, the \emph{avalanche monomial} of $c$ at $v_k$ is
$$\mu_{G} (c, v_k) = \xx^{\nu(c,v_{k})} = \prod_{i=1}^n x_i^{f_i}$$ 
where $\nu(c,v_{k}) = (f_{1},\dots, f_{n})$ is the toppling vector of the
stabilization of $c+1_{v_{k}}$. 
\end{defn}

\begin{defn} Let $G = (V,E,s)$ be a graph. 
The \emph{multivariate avalanche polynomial} of $G$ is defined by 
$$ \sA_G(x_1,\dots,x_n)=\sum_{c\in \sg(G)} \sum_{v\in\rV}
\mu_{G} (c, v).$$  
\end{defn}

Note that the multivariate avalanche polynomial is the sum of all
possible avalanche monomials. In what follows the term ``avalanche
polynomial'' will refer to the multivariate case.

%Working with the same graph as Example \ref{ex:CycleAvalanche}, we
%generate the toppling polynomial for $C_3$. First, we define toppling
%vector.

% \begin{defn} Say $f$ is a \emph{toppling vector} if $f$ is an integer
% vector indexed by the vertices in $\widetilde{V}(G)$ where $f(v)$
% is equal to the number of times vertex $v$ appears in a sequence of
% topplings. Let $c$ and $b$ be stable sandpiles on a graph
% $G$. Define $\nu( c \oplus b)$ to be the toppling vector needed
% to stabilize the sandpile $c+b$.
% \end{defn}

\begin{ex}
\label{ex:CycleToppling} 
As in Example \ref{ex:CycleAvalanche} we have recurrents $(1,0),
(0,1),$ and $(1,1)$ in $C_{3}$. The table in Figure
\ref{table:C3multi} records the toppling 
vector $\nu(c,v_{i})$ for the
principal avalanche of $c$ at $v_i$. 
% for $c\in \sg(C_3)$ and $i=1,2$.

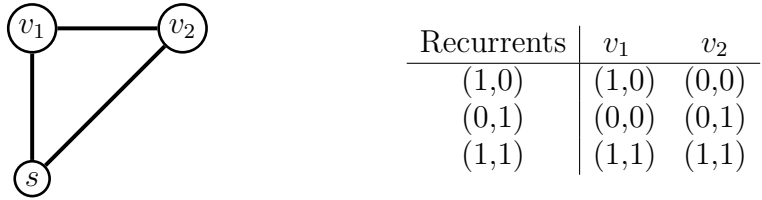
\begin{figure}[!htbp]
\begin{minipage}[c]{0.3\linewidth}
\centering
\begin{tikzpicture} 
\tikzstyle{VertexStyle} = [shape = circle, draw = black,
                                          inner sep = 2pt, outer sep = 0.5pt, minimum size = 2mm,
                                         line width = 1pt]
\SetUpEdge[lw=1.5pt]
\SetGraphUnit{2} 
\tikzset{EdgeStyle/.style={-}}
\Vertex[Math,L=v_1]{1}
\EA[Math,L=v_2](1){2}
\SO[Math,L=s](1){3}
\Edges(1,2,3,1)
\end{tikzpicture}
\end{minipage}
\begin{minipage}[c]{0.45\linewidth}
\centering
\begin{tabular}{ c|c c}
Recurrents & $v_{1}$ \hspace{2pt} & $v_{2}$ \\ \hline
(1,0) & (1,0) & (0,0) \\
(0,1) & (0,0) & (0,1) \\
(1,1) & (1,1) & (1,1)  
\end{tabular}
\end{minipage} 
\caption{Toppling vectors on $C_{3}$.}
\label{table:C3multi}
\end{figure}

From this table we can easily read the avalanche monomials.  For
example, $\mu_{C_3}((1,0),v_1) = x_1^1x_2^0 = x_1$. Adding all
avalanche monomials gives the avalanche polynomial

$$\sA_{C_{3}} (x_1,x_2) = x_1^1 x_2^0 +2 x_1^0 x_2^0 + x_1^0 x_2^1 + 2
x_1^1 x_2^1 =  2x_1 x_2 + x_1 + x_2 + 2.$$ 
\end{ex}

Note that the univariate avalanche polynomial for a graph $G$ can
be recovered from the multivariate avalanche polynomial by
substituting each $x_{i}$ by $x$, i.e.,  
%However, the reverse is not true: if you start with an
%avalanche polynomial for some graph you cannot always determine the
%toppling polynomial.

\[\sA_{G}(x) = \sA_{G}(x,\dots, x).\]

It is a well-known fact that the sandpile group of an undirected graph
is independent of the choice of sink \cite[Proposition
1.1]{CR00}. Nevertheless, the structure of 
the individual recurrent sandpiles may differ. This implies that the
avalanche polynomial of a graph is dependent on the choice of
sink. For this reason, we will always fix a sink before discussing the
avalanche polynomial of a graph. Nevertheless, 
for certain graphs like cycles and complete graphs, the avalanche
polynomial does not depend on the choice of sink. For other families
such as wheel and fan graphs there is a natural choice of sink, namely
the dominating vertex. For trees, our method computes the avalanche
polynomial for any choice of sink.

\subsection{Burst Size}

There are other evaluations of the multivariate avalanche polynomial
$\sA_{G}(x_{1},\dots, x_{n})$ that are 
relevant in the larger field of sandpile groups. 
In \cite{L15}, Levine introduces the concept of \emph{burst size} to
prove a conjecture of Poghosyan, Poghosyan, Priezzhev, Ruelle
\cite{PPPR11} on the relationship between the threshold state of the
fixed-energy sandpile and the stationary state of Dhar's abelian
sandpile. 

\begin{defn}
Let $G = (V,E,s)$ be a graph and $c$ be a recurrent sandpile on $G$. Given $v\in
\rV$, define the \emph{burst size} of $c$ at $v$ as
$$\text{av}(c,v):= |c'| - |c| + 1,$$
where the sandpile $c'$ is defined such that $c'\oplus 1_{v} = c$ and
$|c|$ denotes the number of grains of sand in $c$.
Equivalently, $\text{av}(c,v)$ is the number of grains that fall into
the sink $s$ during the stabilization of $c'+1_{v} \leadsto c$. 
\end{defn}

Let $G = (V,E,s)$ be a \emph{simple graph} and  $\sA_{G}(x_{1},\dots,
x_{n})$ be its multivariate avalanche polynomial. Now, let $x_{i} = 1$
for each vertex $v_{i}$ that is not adjacent to the sink $s$. Also,
for each vertex $v_{j}$ adjacent to $s$, let $x_{j} = x$. 
The resulting univariate polynomial $\sB(x) = \sum_{k} b_{k}x^{k}$
satisifes the condition that $b_{k}$ is the number of principal
avalanches with burst size $k$.

%%%%%%%%%%%%%%%%%%%%TREES%%%%%%%%%%%%%%%%%%%%%%%

\section{Toppling Polynomials of Trees}\label{trees} 

%The first family of graphs we examine are trees. 
Let $T$ be a tree on $n+1$ vertices
labelled $v_{1},\dots, v_{n},s$. Assume further that $T$ is
rooted at the sink $s$. It is a basic observation that $T$ has only 
one recurrent, namely $\smax_{T}$.  So
$$\sA_T(x_{1},\dots, x_{n}) = \sum_{v\in\rV} \mu_T(\smax_{T},v).$$

As noted in \cite{CDR04} any tree can be constructed from a
single vertex using two operations $\phi$ and $+$ defined below.  
%This idea will be crucial in computing the avalanche polynomial of trees.

\begin{defn}
\label{defn:root} For two trees $T$ and $T'$ rooted at $s$ and $s'$
respectively, the operation $+$, called \emph{tree addition},
identifies $s$ and $s'$.  For a tree $T$ rooted at $s$, the operation
$\phi$, called \emph{grafting} or \emph{root extension}, refers to
adding an edge from $s$ to a new root $s'$.
\end{defn}

The operations $\phi$ and $+$ can be seen in Figure \ref{fig:phi+}.
Theorem \ref{thm:ToppOfTrees} explains what happens to the toppling
polynomial of a tree under these operations.

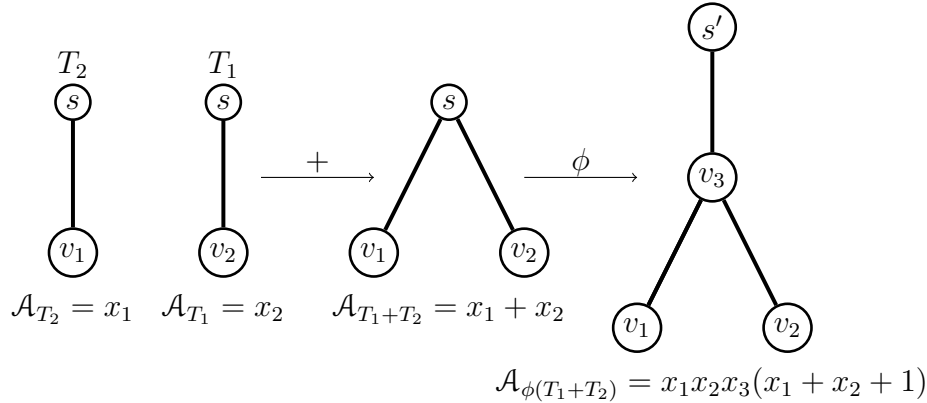
\begin{figure}[!h]
\begin{center}
\begin{tikzpicture}[scale=1]
\tikzstyle{VertexStyle} = [shape = circle, draw = black, 
                                          inner sep = 2pt, outer sep = 0.5pt, minimum size = 2mm,
                                         line width = 1pt]
\SetUpEdge[lw=1.5pt] 
\SetGraphUnit{2} 
\tikzset{EdgeStyle/.style={-}}
\Vertex[Math,L=s]{1}
\node at (0,.5) {$T_1$};
\node at (-2,.5) {$T_2$};
\SO[Math,L=v_2](1){2}
\node at (0,-2.75) {$\sA_{T_1}=x_2$};
\Edges(1,2)
\WE[Math,L=s](1){3}
\SOWE[Math,L=v_1](1){4}
\node at (-2,-2.75) {$\sA_{T_2}=x_1$};
\Edges(3,4)
\draw [->] (.5,-1) -- (2,-1);
\node at (1.25,-0.8) {+};
\coordinate (5) at (2,0);
\coordinate (8) at (1,0);
\EA[Math,L=s](8){7}
\SO[Math,L=v_1](5){6}
\EA[Math,L=v_2](6){9}
\Edges(6,7,9)
\node at (3,-2.75) {$\sA_{T_1 + T_2}=x_1+x_2$};
\draw [->] (4,-1)--(5.5,-1);
\node at (4.75,-0.8) {$\phi$};
\coordinate (11) at (4.5,-1);
\coordinate (14) at (5.5,-1);
\EA[Math,L=v_3](11){12}
\NO[Math,L=s'](12){13}
\SO[Math,L=v_1](14){15}
\EA[Math,L=v_2](15){16}
\Edges(13,12,15,12,16)
\node at (6.5,-3.75) {$\sA_{\phi(T_1 + T_2)}=x_1x_2x_3(x_1+x_2+1)$};
\end{tikzpicture}
\caption{Effect of $+$ and $\phi$ operations on the avalanche polynomial}
\label{fig:phi+}
\end{center}
\end{figure}

\begin{thm}
\label{thm:ToppOfTrees} Let $\sA_{T}$, $\sA_{T_1}$, and $\sA_{T_2}$ be
the avalanche polynomials of trees $T$, $T_1$, and $T_2$,
respectively. Then %the following are true:
\begin{enumerate}
\item $\sA_{T_1+T_2}= \sA_{T_1} +\sA_{T_2}$,
\item $\sA_{\phi(T)}=x_1x_2 \cdots x_n(\sA_{T}+1)$, where 
$n = |\rV(\phi(T))|$.
\end{enumerate}
 \end{thm}

\begin{proof} Under tree addition, trees $T_1$ and $T_2$ are only
connected at the sink $s$. Since the sink never topples, a principal
avalanche at a vertex in $T_1$ will never affect the vertices in
$T_2$, and vice versa. Therefore, $\sA_{T_1+T_2}= \sA_{T_1}
+\sA_{T_2}$.
For the second part, let $T$ be a tree on $n$ vertices with sink $s$. 
%and its root extension $\phi(T)$ on $n+1$ vertices with sink $s$.  
Let $\smax_{\phi(T)}$ and $\smax_{T}$ be
the maximum stable sandpile on $\phi(T)$ and $T$, respectively.  First
we consider $\smax_{\phi(T)} 
\oplus 1_s$.  By Proposition \ref{prop:Burn},
$$\mu_{\phi(T)} (\smax_{\phi(T)}, s) = x_1x_2\cdots x_n.$$
Now consider $\smax_{\phi(T)}\oplus 1_{v_k}$ where $v_k\neq s$. Note
that when we apply the toppling sequence %associated with
$\nu_T(\smax_T,v_k)$ to $\smax_{\phi(T)} + 1_{v_k}$,  we get the
sandpile $\smax_{\phi(T)} + 1_s$.  Thus, 
each vertex will now topple once more. So for each $v_k$ with $v_k\neq
s$, 
$$\mu_{\phi(T)}(\smax_{\phi(T)},v_k) = (x_1x_{2}\cdots x_n ) \cdot
\mu_{T} (\smax_T, v_k).$$ 
Therefore,
$$\sA_{\phi(T)} = (x_1x_2 \cdots x_n) + (x_1x_2 \cdots x_n) \cdot
\sA_{T} =x_1x_2 \cdots x_n(\sA_{T}+1).$$  
\end{proof}

 Note that using Theorem \ref{thm:ToppOfTrees} we can
compute the multivariate avalanche polynomial of any tree.
Furthermore, 
as noted in \cite{CDR04}, it is possible for two non-isomorphic trees
to have the same univariate avalanche polynomial. In contrast, the
multivariate avalanche polynomial distinguishes between labeled trees.

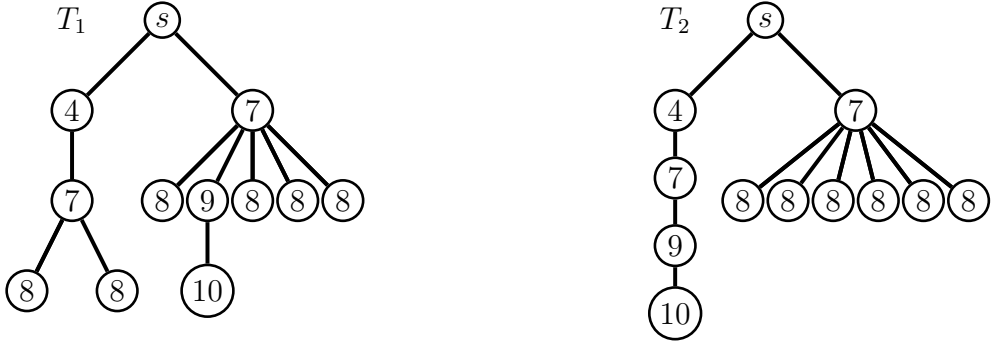
\begin{figure}[!hbpt]
\begin{center}
\begin{multicols}{2}
\begin{tikzpicture}[scale=.6]
\tikzstyle{VertexStyle} = [shape = circle, draw = black, 
                                          inner sep = 2pt, outer sep = 0.5pt, minimum size = 2mm,
                                         line width = 1pt]
\SetUpEdge[lw=1.5pt] 
\SetGraphUnit{2} 
\tikzset{EdgeStyle/.style={-}}
\Vertex[Math,L=s]{0}
\node at (-2,0) {$T_1$};
\SOWE[Math,L=4](0){1}
\SOEA[Math,L=7](0){2}
\SO[Math,L=7](1){3}
\SO[Math,L=8](2){4}
\WE[Math,L=8](4){6}
\EA[Math,L=8](4){8}
\coordinate (5) at (3,-4);
\WE[Math,L=9](5){7}
\EA[Math,L=8](7){10}
\coordinate (9) at (-1,-6);
\WE[Math,L=8](9){11}
\EA[Math,L=8](11){12}
\SO[Math,L=10](7){13}
\Edges(0,1,3,11,3,12)
\Edges(0,2,4,2,10,2,6,2,8,2,7,13)
\end{tikzpicture}

\columnbreak

\begin{tikzpicture}[scale=.6]
\tikzstyle{VertexStyle} = [shape = circle, draw = black, 
                                          inner sep = 2pt, outer sep = 0.5pt, minimum size = 2mm,
                                         line width = 1pt]
\SetUpEdge[lw=1.5pt] 
\SetGraphUnit{2} 
\tikzset{EdgeStyle/.style={-}}
\Vertex[Math,L=s]{0}
\node at (-2,0) {$T_2$};
\SOWE[Math,L=4](0){1}
\coordinate (2) at (-2,-1.5);
\SO[Math,L=7](2){3}
\coordinate (6) at (-2,-3);
\SO[Math,L=9](6){4}
\coordinate (7) at (-2,-4.5);
\SO[Math,L=10](7){5}
\SOEA[Math,L=7](0){8}
\coordinate (9) at (2.5,-4);
\WE[Math,L=8](9){10}
\EA[Math,L=8](10){11}
\EA[Math,L=8](11){12}
\coordinate (13) at (1.5,-4);
\WE[Math,L=8](13){14}
\EA[Math,L=8](14){15}
\EA[Math,L=8](15){16}
\Edges(0,1,3,4,5)
\Edges(0,8,16,8,15,8,14,8,12,8,11,8,10)
\end{tikzpicture}
\end{multicols}
\caption{$\sA_{T_1}(x) = \sA_{T_2}(x) = x^{10}+x^9+6x^8+2x^7+x^4$}
\label{fig:noniso}
\end{center}
\end{figure}

Figure \ref{fig:noniso} gives an example, first presented in
\cite{CDR04}, of two non isomorphic trees with the same univariate
avalanche polynomial. The vertices are labeled with the size of the 
principal avalanche starting at that vertex. One can clearly see that 
$T_1$ and $T_2$ have the same univariate avalanche
polynomial. However, they have different multivariate avalanche
polynomials.  Let's examine the right subtrees of $T_1$ and $T_2$,
denoted by $R_{1}$ and $R_{2}$, respectively. 

\begin{figure}[!h]
\centering
\begin{multicols}{2}

\begin{tikzpicture}[scale=.6]
\tikzstyle{VertexStyle} = [shape = circle, draw = black, 
                                          inner sep = 2pt, outer sep = 0.5pt, minimum size = 2mm,
                                         line width = 1pt]
\SetUpEdge[lw=1.5pt] 
\SetGraphUnit{2} 
\tikzset{EdgeStyle/.style={-}}
\Vertex[Math,L=s]{s}
\node at (-2,0) {$R_1$};
\SOWE[Math,L=v_7](s){v7}
\SO[Math,L=v_8](s){v8}
\SOEA[Math,L=v_9](s){v9}
\EA[Math,L=v_{10}](v9){v10}
\WE[Math,L=v_{6}](v7){v6}
\SO[Math,L=v_5](v7){v5}
\Edges(s,v6,s,v7,v5)
\Edges(s,v8,s,v9,s,v10)
\end{tikzpicture}
%\caption{$\sA=x_6 + [x_5^2x_7 + x_5x_7 ] + x_8 + x_9 +x_{10}$}
%\label{F:subtree1}
    
\columnbreak

\begin{tikzpicture}[scale=.6]
\tikzstyle{VertexStyle} = [shape = circle, draw = black, 
                                          inner sep = 2pt, outer sep = 0.5pt, minimum size = 2mm,
                                         line width = 1pt]
\SetUpEdge[lw=1.5pt] 
\SetGraphUnit{2} 
\tikzset{EdgeStyle/.style={-}}
\Vertex[Math,L=s]{s}
\node at (-2,0) {$R_2$};
\SO[Math,L=v_7](s){v7}
\SOEA[Math,L=v_8](s){v8}
\EA[Math,L=v_9](v8){v9}
\EA[Math,L=v_{10}](v9){v10}
\WE[Math,L=v_{6}](v7){v6}
\WE[Math,L=v_5](v6){v5}
\Edges(v5,s,v6,s,v7,s,v8,s,v9,s,v10)
\end{tikzpicture}
%\caption{$\sA=x_5 + x_6 + x_7 + x_8 + x_9 + x_{10}$}
%\label{F:subtree2}
\end{multicols}
\label{F:subtrees}
\caption{Labelled right subtrees of $T_{1}$ and $T_{2}$.}
\end{figure}
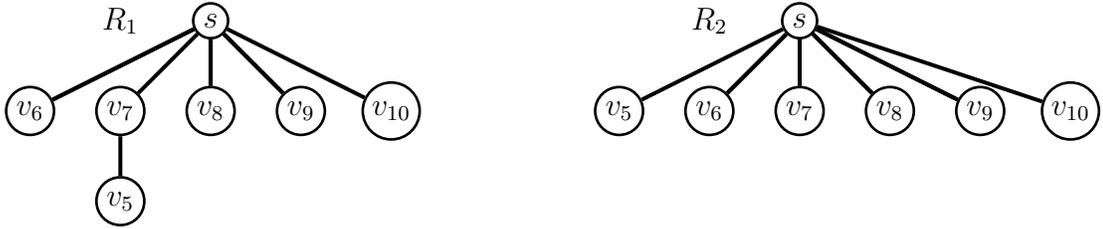 

The avalanche polynomial of $R_{1}$ is
$x_6 + x_5x_7(x_{5}+1) + x_8 + x_9 +x_{10}$ 
and the avalanche polynomial of $R_{2}$ 
is $x_5 + x_6 + x_7 + x_8 + x_9 + x_{10}$. 
Note that the first polynomial has total degree 3 and the second polynomial
has total degree 1. The avalanche polynomials of the left 
subtrees of $T_1$ and $T_2$ will be on  a disjoint set of variables.
Thus, the avalanche polynomials for $T_1$ and $T_2$ must be distinct.

% \begin{cor}
% \label{F:treeonechild} Let $T$ be a tree on $n+1$ vertices with sink
% $s$. 
% Then all the terms of $\sA_{T}(x_{1},\dots, x_{n})$ share the
% common factor $x_1x_{2} \cdots x_n$ if and only if $s$ has exactly one
% child.
%  \end{cor}

% \begin{proof}
% First note that the argument in the preceeding example can be easily
% extended to show that if $s$ has more than one child then the
% avalanche polynomial $\sA_{T}$ is the sum of at least two multivariate 
% polynomials in disjoint sets of variables. Hence 
% $x_1x_{2}\cdots x_n$ cannot be a factor of $\sA_{T}$.
% %
% Conversely, suppose $s$ has exactly one child. By Theorem
% \ref{thm:ToppOfTrees} the avalanche polynomial for $T$ is
% $\sA_T(x_1,...,x_n) = x_1x_{2} \cdots x_n(\sA_{T'}+1)$, where $T'$
% is the tree obtained by deleting $s$, i.e., $T' = T-\{s\}$. Therefore, all
% terms of $\sA_T$ share the common factor $x_1x_{2} \cdots x_n$.
% \end{proof}
 
\begin{cor}\label{cor:PhiIsEffective} 
Let $T$ and $T'$ be  two trees on $n+1$
vertices. Then, 
$\sA_T(x_1,\dots,x_n) = \sA_{T'}(x_1,\dots,x_n)$ if and only if
$\sA_{\phi(T)}(x_1,\dots,x_{n+1}) = \sA_{\phi(T')}(x_1,\dots,x_{n+1})$. 
\end{cor}

\begin{proof} 
Suppose that $\sA_{T} = \sA_{T'}$.  Theorem \ref{thm:ToppOfTrees} implies
\[\sA_{\phi(T)} = x_1x_2 \cdots x_n(\sA_{T}+1) = x_1x_2 \cdots
x_n(\sA_{T'}+1) = \sA_{\phi(T')}.\] 
Now assume $\sA_{\phi(T)} = \sA_{\phi(T')}$. Then $x_1x_2 \cdots x_n(\sA_{T}+1) = x_1x_2 \cdots
x_n(\sA_{T'}+1)$. This clearly implies
$\sA_{T} = \sA_{T'}$. 
\end{proof}

 \begin{thm}\label{uniqueness} 
Let $T$ be a tree on $n+1$ vertices. If $\sA_T(x_1,\dots, x_n) =
\sA_{T'}(x_{1},\dots, x_{n})$ for some tree $T'$,  then $T = T'$.
 \end{thm}

\begin{proof} 
We use induction on the height of $T$.  
Recall that
the \emph{height} of a tree is the number of edges in the longest path
between the root and a leaf.  
Suppose $T$ has height $0$, that is, $T$ consists of one vertex. Then
$\sA_{T} = 0$. Clearly, if $T'$ has two or more vertices then
$\sA_{T'} \neq 0$. Since $\sA_{T} = \sA_{T'}$, then $T'$ must also
consist of one vertex and $T=T'$. 
Now suppose that $T$ has height $h>0$. In this case, the sink $s$ of $T$
must have at least one child. Assume $s$ has degree $d$. 
Deleting $s$ creates
$d$ trees $T_1,T_2,\dots,T_{d}$. We have that $T = \phi(T_{1}) +
\phi(T_{2}) + \dotsm + \phi(T_{d})$. So $\sA_{T}$ is the sum of $d$
multivariate polynomials with pairwise disjoint supports
\[\sA_{T} = \sA_{\phi(T_{1})} +  \sA_{\phi(T_{2})} + \dotsm +
\sA_{\phi(T_{d})}. \] 
Since
$\sA_{T} = \sA_{T'}$, then $\sA_{T'}$ must also satisfy the same
condition. Hence the sink of $T'$ must also have degree $d$ and 
$T' = \phi(T'_{1}) + \phi(T'_{2}) + \dotsm + \phi(T'_{d})$ for some
trees $T'_1,T'_2,\dots,T'_{d}$. Since the supports are pairwise
disjoint, we must also have that
$\sA_{\phi(T_{1})} = \sA_{\phi(T'_{j})}$, for some $j$. Corollary
\ref{cor:PhiIsEffective} implies 
$\sA_{T_{1}} = \sA_{T'_{j}}$. But $T_{1}$ is a tree of height
$h-1$. By induction, $T_{1} = T'_{j}$. Therefore,
after relabeling the subtrees in $T'$, we must have $T_{i} = T'_{i}$
for all $1\leq i \leq d$ and $T = T'$.
\end{proof}

%%%%%%%%%%%%%%%%%%CYCLES%%%%%%%%%%%%%%%%%%%%%%%%%

\section{Toppling Polynomials of Cycles}\label{cycles}

%Now, we look at the toppling polynomial for cycles on $n+1$ vertices,
%denoted $C_{n+1}$.  
Now, we will compute the avalanche
polynomial of the cycle graph $C_{n+1}$ on $n+1$ vertices. Unless
otherwise stated, we will label the vertices $s, v_1, v_2, \dots, v_n$
in a clockwise manner.  As shown in Example \ref{ex:CycleAvalanche},
%the avalanche polynomial of $C_3$ is 
$\sA_{C_{3}}(x_{1},x_{2}) = 2x_1x_2 + x_1 + x_2 + 2$. We will
denote by $C_{2}$ the graph with two vertices and two edges
between these vertices. 
%We will consider $C_{2}$ as the $2$-cycle graph. 
It is clear that $\sA_{C_{2}}(x_{1}) = x_{1} + 1$.
% Below, we will assume that $n\geq 3$.

In this section, we will write sandpiles and toppling vectors
as strings instead of vectors. For example, the string
$1^{p-1}01^{n-p}$ denotes the sandpile with no grains of sand at
vertex $v_p$ and $1$ grain of sand at every other vertex. We will also
make the convention that a bit raised to the $0$ power does not appear
in the string, e.g., $0^{2}1^{0}0^{2} = 0^{4}$. In the 
previous section, we saw that a tree has exactly one recurrent
sandpile. The cycle graph $C_{n+1}$ has exactly
$n+1$ recurrent sandpiles, namely, $\smax = 1^n$ and $b_{p} =
1^{p-1}01^{n-p}$ for $p = 1, 2, \dots, n$, see \cite{CDR04}.
 
%\begin{lemma}
%The recurrent sandpiles on $C_{n+1}$ are $1^n$ and $1^p01^q$, with $p+q = n-1$.
%\end{lemma}

\subsection{Toppling Sequence for the Maximal Stable Sandpile} We
first focus our attention on understanding the toppling sequences for
$1^n + 1_{v_i}$ for $1\leq i \leq n$.

\begin{ex} \label{ex:cycle:1} 
Figure \ref{cycletopple6} shows that
$\mu_{C_6}(1^5, v_2) = x_1x_2^2x_3^2x_4^2x_5$.
\end{ex}

\begin{figure}[!h]

\centering
\begin{multicols}{4}
\begin{tikzpicture}[scale=0.65]
\tikzstyle{VertexStyle} = [shape = circle, draw = black, 
                           fill=white, inner sep = 1pt, outer sep = 0.5pt, minimum size = 2mm, 
                           line width = 1pt]
\SetUpEdge[lw=1.5pt] 
\SetGraphUnit{2} 
\tikzset{EdgeStyle/.style={-}}
 \draw (1,0) arc (0:360: 1 cm);
   \Vertex[a=90 , d=1 cm, Math, L=s]{0}
   \Vertex[a=30 , d=1 cm, Math, L=1]{1}
   \Vertex[a=330 , d=1 cm, Math, L=1]{2}
   \Vertex[a=270 , d=1 cm, Math, L=1]{3}
   \Vertex[a=210 , d=1 cm, Math, L=1]{4}
   \Vertex[a=150 , d=1 cm, Math, L=1]{5}
 \draw [->] (2,0)--(3,0);
 
\tikzstyle{VertexStyle}=[draw=none]
   \Vertex[a=30 , d=1.7 cm, Math, L=v_1]{v1}
   \Vertex[a=330 , d=1.7 cm, Math, L=v_2]{v2}
   \Vertex[a=270 , d=1.7 cm, Math, L=v_3]{v3}
   \Vertex[a=210 , d=1.7 cm, Math, L=v_4]{v4}
   \Vertex[a=150 , d=1.7 cm, Math, L=v_5]{v5}
\end{tikzpicture}

\columnbreak

\begin{tikzpicture}[scale=0.65]
\tikzstyle{VertexStyle} = [shape = circle, draw = black, 
                           fill=white, inner sep = 1pt, outer sep = 0.5pt, minimum size = 2mm, 
                           line width = 1pt]
\SetUpEdge[lw=1.5pt] 
\SetGraphUnit{2} 
\tikzset{EdgeStyle/.style={-}}
 \draw (1,0) arc (0:360: 1 cm);
   \Vertex[a=90 , d=1 cm, Math, L=s]{0}
   \Vertex[a=30 , d=1 cm, Math, L=1]{1}
   \Vertex[a=270 , d=1 cm, Math, L=1]{3}
   \Vertex[a=210 , d=1 cm, Math, L=1]{4}
   \Vertex[a=150 , d=1 cm, Math, L=1]{5}
 \draw [->] (2,0)--(3,0);
 
 \tikzstyle{VertexStyle}=[shape = circle, draw = black, 
                           fill=gray!50!white, inner sep = 1pt, outer
                           sep = 0.5pt, minimum size = 2mm, 
                           line width = 1pt]
\Vertex[a=330 , d=1 cm, Math, L=2]{2}
 
\tikzstyle{VertexStyle}=[draw=none]
   \Vertex[a=30 , d=1.7 cm, Math, L=v_1]{v1}
   \Vertex[a=330 , d=1.7 cm, Math, L=v_2]{v2}
   \Vertex[a=270 , d=1.7 cm, Math, L=v_3]{v3}
   \Vertex[a=210 , d=1.7 cm, Math, L=v_4]{v4}
   \Vertex[a=150 , d=1.7 cm, Math, L=v_5]{v5}
\end{tikzpicture}

\columnbreak

\begin{tikzpicture}[scale=0.65]
\tikzstyle{VertexStyle} = [shape = circle, draw = black, 
                           fill=white, inner sep = 1pt, outer sep = 0.5pt, minimum size = 2mm, 
                           line width = 1pt]
\SetUpEdge[lw=1.5pt] 
\SetGraphUnit{2} 
\tikzset{EdgeStyle/.style={-}}
 \draw (1,0) arc (0:360: 1 cm);
   \Vertex[a=90 , d=1 cm, Math, L=s]{0}
   \Vertex[a=330 , d=1 cm, Math, L=0]{2}
   \Vertex[a=210 , d=1 cm, Math, L=1]{4}
   \Vertex[a=150 , d=1 cm, Math, L=1]{5}
 \draw [->] (2,0)--(3,0);
 
 \tikzstyle{VertexStyle}=[shape = circle, draw = black, 
                           fill=gray!50!white, inner sep = 1pt, outer sep= 0.5pt, minimum size = 2mm, 
                           line width = 1pt]
   \Vertex[a=30 , d=1 cm, Math, L=2]{1}
   \Vertex[a=270 , d=1 cm, Math, L=2]{3}
 
\tikzstyle{VertexStyle}=[draw=none]
   \Vertex[a=30 , d=1.7 cm, Math, L=v_1]{v1}
   \Vertex[a=330 , d=1.7 cm, Math, L=v_2]{v2}
   \Vertex[a=270 , d=1.7 cm, Math, L=v_3]{v3}
   \Vertex[a=210 , d=1.7 cm, Math, L=v_4]{v4}
   \Vertex[a=150 , d=1.7 cm, Math, L=v_5]{v5}
\end{tikzpicture}

\columnbreak

\begin{tikzpicture}[scale=0.65]
\tikzstyle{VertexStyle} = [shape = circle, draw = black, 
                           fill=white, inner sep = 1pt, outer sep = 0.5pt, minimum size = 2mm, 
                           line width = 1pt]
\SetUpEdge[lw=1.5pt] 
\SetGraphUnit{2} 
\tikzset{EdgeStyle/.style={-}}
 \draw (1,0) arc (0:360: 1 cm);
   \Vertex[a=90 , d=1 cm, Math, L=s]{0}
   \Vertex[a=30 , d=1 cm, Math, L=0]{1}
   \Vertex[a=270 , d=1 cm, Math, L=0]{3}
   \Vertex[a=150 , d=1 cm, Math, L=1]{5}

 \tikzstyle{VertexStyle}=[shape = circle, draw = black, 
                           fill=gray!50!white, inner sep = 1pt, outer
                           sep = 0.5pt, minimum size = 2mm, 
                           line width = 1pt]
   \Vertex[a=210 , d=1 cm, Math, L=2]{4}
   \Vertex[a=330 , d=1 cm, Math, L=2]{2}
 
\tikzstyle{VertexStyle}=[draw=none]
   \Vertex[a=30 , d=1.7 cm, Math, L=v_1]{v1}
   \Vertex[a=330 , d=1.7 cm, Math, L=v_2]{v2}
   \Vertex[a=270 , d=1.7 cm, Math, L=v_3]{v3}
   \Vertex[a=210 , d=1.7 cm, Math, L=v_4]{v4}
   \Vertex[a=150 , d=1.7 cm, Math, L=v_5]{v5}
\end{tikzpicture}
\end{multicols}

%LINE2
\begin{multicols}{3}

\begin{tikzpicture}[scale=0.65]
\draw [->] (-5,0)--(-3,0);
%\draw [->] (2,0)--(3,0);
\tikzstyle{VertexStyle} = [shape = circle, draw = black, 
                           fill=white, inner sep = 1pt, outer sep = 0.5pt, minimum size = 2mm, 
                           line width = 1pt]
\SetUpEdge[lw=1.5pt] 
\SetGraphUnit{2} 
\tikzset{EdgeStyle/.style={-}}
 \draw (1,0) arc (0:360: 1 cm);
   \Vertex[a=90 , d=1 cm, Math, L=s]{0}
   \Vertex[a=30 , d=1 cm, Math, L=1]{1}
   \Vertex[a=330 , d=1 cm, Math, L=0]{2}
   \Vertex[a=210 , d=1 cm, Math, L=0]{4}

 \tikzstyle{VertexStyle}=[shape = circle, draw = black, 
                           fill=gray!50!white, inner sep = 1pt, outer
                           sep = 0.5pt, minimum size = 2mm, 
                           line width = 1pt]
   \Vertex[a=150 , d=1 cm, Math, L=2]{5}
   \Vertex[a=270 , d=1 cm, Math, L=2]{3}
\tikzstyle{VertexStyle}=[draw=none]
   \Vertex[a=30 , d=1.7 cm, Math, L=v_1]{v1}
   \Vertex[a=330 , d=1.7 cm, Math, L=v_2]{v2}
   \Vertex[a=270 , d=1.7 cm, Math, L=v_3]{v3}
   \Vertex[a=210 , d=1.7 cm, Math, L=v_4]{v4}
   \Vertex[a=150 , d=1.7 cm, Math, L=v_5]{v5}
\end{tikzpicture}

\columnbreak

\begin{tikzpicture}[scale=0.65]
\tikzstyle{VertexStyle} = [shape = circle, draw = black, 
                           fill=white, inner sep = 1pt, outer sep = 0.5pt, minimum size = 2mm, 
                           line width = 1pt]
\SetUpEdge[lw=1.5pt] 
\SetGraphUnit{2} 
\tikzset{EdgeStyle/.style={-}}
 \draw (1,0) arc (0:360: 1 cm);
   \Vertex[a=90 , d=1 cm, Math, L=s]{0}
   \Vertex[a=30 , d=1 cm, Math, L=1]{1}
   \Vertex[a=330 , d=1 cm, Math, L=1]{2}
   \Vertex[a=270 , d=1 cm, Math, L=0]{3}
   \Vertex[a=150 , d=1 cm, Math, L=0]{5}
 \draw [->] (-5,0)--(-3,0);
 
 \tikzstyle{VertexStyle}=[shape = circle, draw = black, 
                           fill=gray!50!white, inner sep = 1pt, outer sep= 0.5pt, minimum size = 2mm, 
                           line width = 1pt]
   \Vertex[a=210 , d=1 cm, Math, L=2]{4}
 
\tikzstyle{VertexStyle}=[draw=none]
   \Vertex[a=30 , d=1.7 cm, Math, L=v_1]{v1}
   \Vertex[a=330 , d=1.7 cm, Math, L=v_2]{v2}
   \Vertex[a=270 , d=1.7 cm, Math, L=v_3]{v3}
   \Vertex[a=210 , d=1.7 cm, Math, L=v_4]{v4}
   \Vertex[a=150 , d=1.7 cm, Math, L=v_5]{v5}
\end{tikzpicture}

\columnbreak

\begin{tikzpicture}[scale=0.65]
\tikzstyle{VertexStyle} = [shape = circle, draw = black, 
                           fill=white, inner sep = 1pt, outer sep = 0.5pt, minimum size = 2mm, 
                           line width = 1pt]
\SetUpEdge[lw=1.5pt] 
\SetGraphUnit{2} 
\tikzset{EdgeStyle/.style={-}}
 \draw (1,0) arc (0:360: 1 cm);
   \Vertex[a=90 , d=1 cm, Math, L=s]{0}
   \Vertex[a=30 , d=1 cm, Math, L=1]{1}
   \Vertex[a=330 , d=1 cm, Math, L=1]{2}
   \Vertex[a=270 , d=1 cm, Math, L=1]{3}
   \Vertex[a=210 , d=1 cm, Math, L=0]{4}
   \Vertex[a=150 , d=1 cm, Math, L=1]{5}
 \draw [->] (-5,0)--(-3,0);
\tikzstyle{VertexStyle}=[draw=none]
   \Vertex[a=30 , d=1.7 cm, Math, L=v_1]{v1}
   \Vertex[a=330 , d=1.7 cm, Math, L=v_2]{v2}
   \Vertex[a=270 , d=1.7 cm, Math, L=v_3]{v3}
   \Vertex[a=210 , d=1.7 cm, Math, L=v_4]{v4}
   \Vertex[a=150 , d=1.7 cm, Math, L=v_5]{v5}
\end{tikzpicture}
\end{multicols}
\caption{The principal avalanche created by adding a grain of sand to
  $v_2$.}
% The shaded vertices are unstable.}
\label{cycletopple6}
\end{figure}
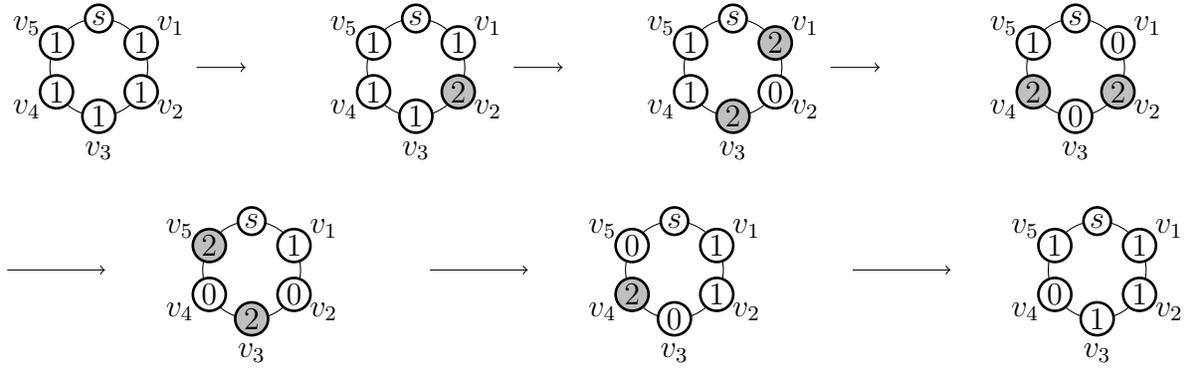

\begin{lemma}
\label{lem:SinkVertices} Let $n\geq 1$.  Then
$$\mu_{C_{n+1}}(1^n, v_1) = \mu_{C_{n+1}}(1^n,v_n) = x_1x_{2}\cdots x_n.$$
\end{lemma}
\begin{proof} Because $v_1$ and $v_n$ are adjacent to the sink, each
vertex topples at most once by Corollary \ref{lem:ToppleOnce}.  Now
consider $\smax + 1_{v_1}$.  
%Adding a grain of sand to $v_1$ makes $v_1$ unstable.  
Toppling $v_1$ results in $v_2$ being
unstable. Inductively, for $i\geq 2$, if $v_i$ becomes unstable it
will topple and $v_{i+1}$ will become unstable.  
%So each vertex must topple at least once.  
Thus each vertex topples exactly once. A similar argument works for
$\mu_{C_{n+1}}(1^n,v_n)$. 
\end{proof}

%\begin{defn}
%Define a \emph{toppling stage} as a maximal sequence of topplings in which each vertex is toppled at most once. \end{defn}
% Note a sandpile may not be stable at the completion of a toppling stage.

%\begin{rmk} 
Observe that the reduced Laplacian of the cycle $C_{n+1}$ is the
$n\times n$ matrix
$$\widetilde{L} = \begin{bmatrix*}[r]
2 &-1 & & & &\\ -1 &2 &-1 & & & \\ & \ddots &\ddots &\ddots & & \\ &
&-1 &2 &-1 & \\ & & &-1 &2
\end{bmatrix*},$$ with 2's on the diagonal, $-1$'s on the off
diagonals, and 0's elsewhere.  
We will use the reduced Laplacian in the next two proofs. We will
abuse notation and write $\rL b$ instead of $\rL b^{t}$ to denote the
product of $\rL$ times the vector $b$.
%\end{rmk}

\begin{lemma}
\label{lem:FirstStage} Let $n\geq 3$, $2\leq i \leq n-1$ and consider the
sandpile $1^n + 1_{v_i}$ on $C_{n+1}$.  We can legally topple each
vertex once and the resulting sandpile is $01^{i-2}21^{n-i-1}0$.
\end{lemma}

\begin{proof} The sequence $v_i,v_{i-1},v_{i-2},\dots,
v_1,v_{i+1},v_{i+2},\dots,v_{n}$ is a legal toppling sequence.  
%Let $\widetilde{L}$ be the reduced Laplacian of $C_{n+1}$.  
By Proposition
\ref{prop:laplacian}, $1^n +1_{v_i}$ accesses the sandpile
$$(1^n +1_{v_i})  - \widetilde{L}\cdot 1^n = (1^n + 1_{v_i}) -
10^{n-2}1 = 01^{i-2}21^{n-i-1}0.$$ 
%Thus, $\max + 1_{v_i}\leadsto 01^{i-2}21^{n-i-1}0$.
\end{proof}

Lemma \ref{lem:FirstStage} states that 
% for $2\leq i\leq n-1$ and $n\geq 3$, 
$1^n +1_{v_i}\leadsto 01^{i-2}21^{n-i-1}0$, another
unstable sandpile. Lemma \ref{lemma:OtherStages} will describe the
remaining toppling pattern.

\begin{lemma} \label{lemma:OtherStages}
Let $n\geq 3$, $2\leq i\leq n-1$, and $1\leq k
\leq \min\{i-1,n-i\}$. Let $c_{k}$ be the sandpile 
$$c_k = 1^{k-1}01^{i-k-1}21^{n-i-k}01^{k-1}.$$
% on $C_{n+1}$. 
Then for each $1\leq k\leq \min\{i-2,n-i-1\}$ the sandpile
$c_{k}$ accesses the sandpile  $c_{k+1}$
via toppling vertices $v_{k+1},v_{k+2},\dots,v_{n-k}$.
\end{lemma}

\begin{proof} The toppling sequence $v_i,v_{i-1},\dots, v_{k+1},
v_{i+1},\dots,v_{n-k}$ is a legal toppling sequence. Applying this
toppling sequence we obtain 
$\widetilde{L}\cdot 0^{k}1^{n-2k}0^{k} =
0^{k-1}(-1)10^{n-2k-2}1(-1)0^{k-1}$. Thus, Proposition
\ref{prop:laplacian} implies $c_{k}$ accesses the sandpile

\[
1^{k-1}01^{i-k-1}21^{n-i-k}01^{k-1} - 0^{k-1}(-1)10^{n-2k-2}1(-1)0^{k-1} 
= 1^k01^{i-k-2}21^{n-i-k-1}01^k = c_{k+1}.
\]
\end{proof}

\begin{thm}
\label{thm:CycleTopplingMonomialsMax} For $n\geq 1$, $1\leq i \leq n$,
let $m = \min\{i,n-i+1\}$.  Then
 $$\mu_{C_{n+1}} (1^n, v_i) = \prod_{j=1}^{m} x_j\cdots
 x_{n-j+1}=(x_1\cdots x_n)(x_2\cdots x_{n-1})\cdots(x_m\cdots
 x_{n-m+1}).$$  
\end{thm}

\begin{proof}
The cases for $n=1$ and $n=2$ are discussed in the first paragraph of
this section. Now assume $n\geq 3$.
If $i = 1$ or $i=n$, the result follows from Lemma
\ref{lem:SinkVertices}.  Suppose $2 \leq i \leq n-1$, by Lemma
\ref{lem:FirstStage}  $\max+1_{v_i} \leadsto c_1$ via the toppling
of all vertices which gives the factor $x_1x_2\cdots
x_n$.  By Lemma \ref{lemma:OtherStages},
$c_1 \leadsto c_2 \leadsto \cdots \leadsto c_{m-1}.$
% where $m = \min\{i,n+1-i\}$. 
Note that $c_{k}\leadsto c_{k+1}$ via
the toppling of vertices $v_{k+1},\dots, v_{n-k}$. This produces the factor
$x_{k+1}\cdots x_{n-k}$. 
Suppose that $m=i=\min\{i,n-i+1\}$. Then
$c_{m-1} = c_{i-1} = 1^{i-2}021^{n-2i+1}01^{i-2}$.  Only vertex $v_i$
is unstable, but when $v_i$ topples $v_{i-1}$ does not become
unstable.  Vertex $v_{i+1}$ does become unstable and will topple.  In
fact, vertices $v_i, v_{i+1}, \dots, v_{n-i+1}$ will topple resulting 
in the stable sandpile $1^{n-i}01^{i-1}$.  This gives the
last factor $x_m\cdots x_{n-m+1}$.  So, if
$m= i = \min\{i,n-i+1\}$, 
$$\mu_{C_{n+1}}(1^n,v_i) = \prod_{j=1}^m x_j\dotsm x_{n-j+1}.$$
The proof is similar for the case $m = n-i+1 = \min\{i,n-i+1\}$.
\end{proof}

\subsection{The toppling sequence for recurrents $1^{p-1}01^{n-p}$}

Theorem \ref{thm:CycleTopplingMonomialsMax} gives us the avalanche
monomials for the maximal stable sandpile at all vertices. Now we find
the avalanche monomials for recurrents of the form $1^{p-1}01^{n-p}$.
We will see that these monomials are closely related to the
avalanche monomials arising from $1^n$.

\begin{ex} \label{ex:cycle:2} In Example \ref{ex:cycle:1}, we saw
$\mu_{C_6} (1^5,v_2) = x_1x_2^2x_3^2x_4^2x_5$. Figure \ref{1p01q}
shows that $\mu_{C_{10}}(1^301^5,v_6) = x_5x_6^2x_7^2x_8^2x_9$. Notice
that the structure of these monomials. 
The only difference is that there is a relabeling of the
variables $x_i \to x_{i+4}$. %, with $p=4$ in this case.
\end{ex}

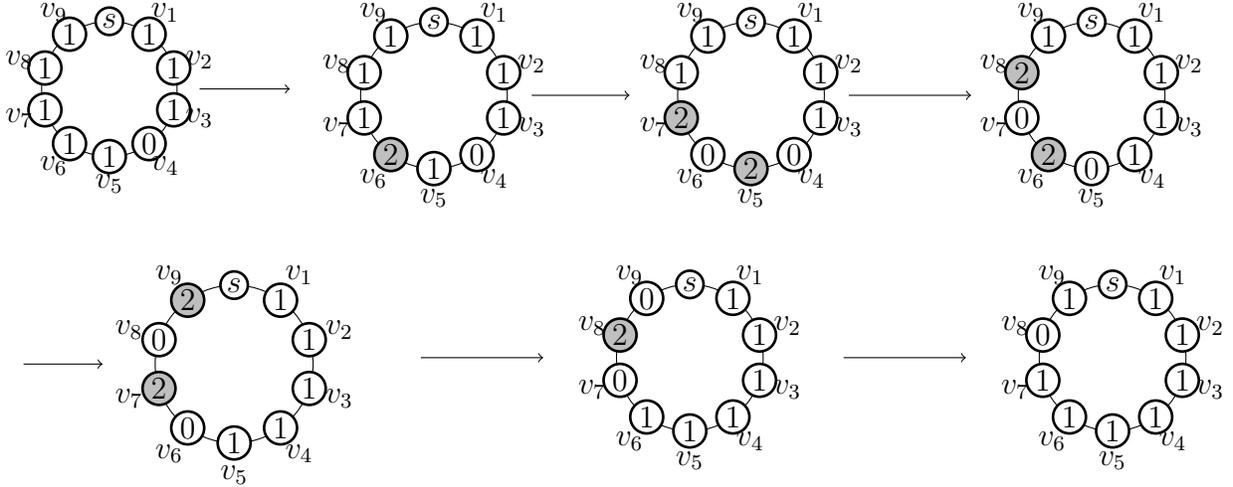
\begin{figure}[!h]
\centering
\begin{multicols}{4}

\begin{tikzpicture}[scale=0.6]
\tikzstyle{VertexStyle} = [shape = circle, draw = black, 
                           fill=white, inner sep = 1pt, outer sep = 0.5pt, minimum size = 2mm, 
                           line width = 1pt]
\SetUpEdge[lw=1.5pt] 
\SetGraphUnit{2} 
\tikzset{EdgeStyle/.style={-}}
 \draw (1.5,0) arc (0:360: 1.5 cm);
    \Vertex[a=90 , d=1.5 cm, Math, L=s]{0}
   \Vertex[a=54 , d=1.5 cm, Math, L=1]{1}
   \Vertex[a=18 , d=1.5 cm, Math, L=1]{2}
   \Vertex[a=-18 , d=1.5 cm, Math, L=1]{3}
   \Vertex[a=-54 , d=1.5 cm, Math, L=0]{4}
   \Vertex[a=270 , d=1.5 cm, Math, L=1]{5}
   \Vertex[a=234 , d=1.5 cm, Math, L=1]{6}
   \Vertex[a=198 , d=1.5 cm, Math, L=1]{7}
   \Vertex[a=-198 , d=1.5 cm, Math, L=1]{8}
   \Vertex[a=-234 , d=1.5 cm, Math, L=1]{9}
\tikzstyle{VertexStyle}=[draw=none]
   \Vertex[a=54 , d=2.1 cm, Math, L=v_1]{v1}
   \Vertex[a=18 , d=2.1 cm, Math, L=v_2]{v2}
   \Vertex[a=-18 , d=2.1 cm, Math, L=v_3]{v3}
   \Vertex[a=-54 , d=2.1 cm, Math, L=v_4]{v4}
   \Vertex[a=270 , d=2.1 cm, Math, L=v_5]{v5}
   \Vertex[a=234 , d=2.1 cm, Math, L=v_6]{v6}
   \Vertex[a=198 , d=2.1 cm, Math, L=v_7]{v7}
   \Vertex[a=-198 , d=2.1 cm, Math, L=v_8]{v8}
   \Vertex[a=-234 , d=2.1 cm, Math, L=v_9]{v9}
\draw[->] (2,0)--(4,0);
\end{tikzpicture}

\columnbreak

\begin{tikzpicture}[scale=0.65]
\tikzstyle{VertexStyle} = [shape = circle, draw = black, 
                           fill=white, inner sep = 1pt, outer sep = 0.5pt, minimum size = 2mm, 
                           line width = 1pt]
\SetUpEdge[lw=1.5pt] 
\SetGraphUnit{2} 
\tikzset{EdgeStyle/.style={-}}
 \draw (1.5,0) arc (0:360: 1.5 cm);
    \Vertex[a=90 , d=1.5 cm, Math, L=s]{0}
   \Vertex[a=54 , d=1.5 cm, Math, L=1]{1}
   \Vertex[a=18 , d=1.5 cm, Math, L=1]{2}
   \Vertex[a=-18 , d=1.5 cm, Math, L=1]{3}
   \Vertex[a=-54 , d=1.5 cm, Math, L=0]{4}
   \Vertex[a=270 , d=1.5 cm, Math, L=1]{5}

   \Vertex[a=198 , d=1.5 cm, Math, L=1]{7}
   \Vertex[a=-198 , d=1.5 cm, Math, L=1]{8}
   \Vertex[a=-234 , d=1.5 cm, Math, L=1]{9}
\tikzstyle{VertexStyle} = [shape = circle, draw = black, 
                           fill=gray!50!white, inner sep = 1pt, outer sep = 0.5pt, minimum size = 2mm, 
                           line width = 1pt]
   \Vertex[a=234 , d=1.5 cm, Math, L=2]{6}                           
\tikzstyle{VertexStyle}=[draw=none]
   \Vertex[a=54 , d=2.1 cm, Math, L=v_1]{v1}
   \Vertex[a=18 , d=2.1 cm, Math, L=v_2]{v2}
   \Vertex[a=-18 , d=2.1 cm, Math, L=v_3]{v3}
   \Vertex[a=-54 , d=2.1 cm, Math, L=v_4]{v4}
   \Vertex[a=270 , d=2.1 cm, Math, L=v_5]{v5}
   \Vertex[a=234 , d=2.1 cm, Math, L=v_6]{v6}
   \Vertex[a=198 , d=2.1 cm, Math, L=v_7]{v7}
   \Vertex[a=-198 , d=2.1 cm, Math, L=v_8]{v8}
   \Vertex[a=-234 , d=2.1 cm, Math, L=v_9]{v9}
\draw[->] (2,0)--(4,0);
\end{tikzpicture}

\columnbreak

\begin{tikzpicture}[scale=0.65]
\tikzstyle{VertexStyle} = [shape = circle, draw = black, 
                           fill=white, inner sep = 1pt, outer sep = 0.5pt, minimum size = 2mm, 
                           line width = 1pt]
\SetUpEdge[lw=1.5pt] 
\SetGraphUnit{2} 
\tikzset{EdgeStyle/.style={-}}
 \draw (1.5,0) arc (0:360: 1.5 cm);
    \Vertex[a=90 , d=1.5 cm, Math, L=s]{0}
   \Vertex[a=54 , d=1.5 cm, Math, L=1]{1}
   \Vertex[a=18 , d=1.5 cm, Math, L=1]{2}
   \Vertex[a=-18 , d=1.5 cm, Math, L=1]{3}
   \Vertex[a=-54 , d=1.5 cm, Math, L=0]{4}

   \Vertex[a=234 , d=1.5 cm, Math, L=0]{6}

   \Vertex[a=-198 , d=1.5 cm, Math, L=1]{8}
   \Vertex[a=-234 , d=1.5 cm, Math, L=1]{9}
\tikzstyle{VertexStyle} = [shape = circle, draw = black, 
                           fill=gray!50!white, inner sep = 1pt, outer sep = 0.5pt, minimum size = 2mm, 
                           line width = 1pt]
   \Vertex[a=198 , d=1.5 cm, Math, L=2]{7}
  \Vertex[a=270 , d=1.5 cm, Math, L=2]{5}
\tikzstyle{VertexStyle}=[draw=none]
   \Vertex[a=54 , d=2.1 cm, Math, L=v_1]{v1}
   \Vertex[a=18 , d=2.1 cm, Math, L=v_2]{v2}
   \Vertex[a=-18 , d=2.1 cm, Math, L=v_3]{v3}
   \Vertex[a=-54 , d=2.1 cm, Math, L=v_4]{v4}
   \Vertex[a=270 , d=2.1 cm, Math, L=v_5]{v5}
   \Vertex[a=234 , d=2.1 cm, Math, L=v_6]{v6}
   \Vertex[a=198 , d=2.1 cm, Math, L=v_7]{v7}
   \Vertex[a=-198 , d=2.1 cm, Math, L=v_8]{v8}
   \Vertex[a=-234 , d=2.1 cm, Math, L=v_9]{v9}
\draw[->] (2,0)--(4.5,0);
\end{tikzpicture}

\columnbreak

\begin{tikzpicture}[scale=0.65]
\tikzstyle{VertexStyle} = [shape = circle, draw = black, 
                           fill=white, inner sep = 1pt, outer sep = 0.5pt, minimum size = 2mm, 
                           line width = 1pt]
\SetUpEdge[lw=1.5pt] 
\SetGraphUnit{2.0} 
\tikzset{EdgeStyle/.style={-}}
 \draw (1.5,0) arc (0:360: 1.5 cm);
    \Vertex[a=90 , d=1.5 cm, Math, L=s]{0}
   \Vertex[a=54 , d=1.5 cm, Math, L=1]{1}
   \Vertex[a=18 , d=1.5 cm, Math, L=1]{2}
   \Vertex[a=-18 , d=1.5 cm, Math, L=1]{3}
   \Vertex[a=-54 , d=1.5 cm, Math, L=1]{4}
  \Vertex[a=270 , d=1.5 cm, Math, L=0]{5}
   \Vertex[a=198 , d=1.5 cm, Math, L=0]{7}
   \Vertex[a=-234 , d=1.5 cm, Math, L=1]{9}
\tikzstyle{VertexStyle} = [shape = circle, draw = black, 
                           fill=gray!50!white, inner sep = 1pt, outer sep = 0.5pt, minimum size = 2mm, 
                           line width = 1pt]
   \Vertex[a=-198 , d=1.5 cm, Math, L=2]{8}
   \Vertex[a=234 , d=1.5 cm, Math, L=2]{6}
\tikzstyle{VertexStyle}=[draw=none]
   \Vertex[a=54 , d=2.1 cm, Math, L=v_1]{v1}
   \Vertex[a=18 , d=2.1 cm, Math, L=v_2]{v2}
   \Vertex[a=-18 , d=2.1 cm, Math, L=v_3]{v3}
   \Vertex[a=-54 , d=2.1 cm, Math, L=v_4]{v4}
   \Vertex[a=270 , d=2.1 cm, Math, L=v_5]{v5}
   \Vertex[a=234 , d=2.1 cm, Math, L=v_6]{v6}
   \Vertex[a=198 , d=2.1 cm, Math, L=v_7]{v7}
   \Vertex[a=-198 , d=2.1 cm, Math, L=v_8]{v8}
   \Vertex[a=-234 , d=2.1 cm, Math, L=v_9]{v9}
\end{tikzpicture}
\end{multicols}

%LINE2
\begin{multicols}{3}
\begin{tikzpicture}[scale=0.7]
\draw [->] (-4,0)--(-2.5,0);
\tikzstyle{VertexStyle} = [shape = circle, draw = black, 
                           fill=white, inner sep = 1pt, outer sep = 0.5pt, minimum size = 2mm, 
                           line width = 1pt]
\SetUpEdge[lw=1.5pt] 
\SetGraphUnit{2} 
\tikzset{EdgeStyle/.style={-}}
 \draw (1.5,0) arc (0:360: 1.5 cm);
    \Vertex[a=90 , d=1.5 cm, Math, L=s]{0}
   \Vertex[a=54 , d=1.5 cm, Math, L=1]{1}
   \Vertex[a=18 , d=1.5 cm, Math, L=1]{2}
   \Vertex[a=-18 , d=1.5 cm, Math, L=1]{3}
   \Vertex[a=-54 , d=1.5 cm, Math, L=1]{4}
  \Vertex[a=270 , d=1.5 cm, Math, L=1]{5}
 
   \Vertex[a=234 , d=1.5 cm, Math, L=0]{6}
  \Vertex[a=-198 , d=1.5 cm, Math, L=0]{8}
  
\tikzstyle{VertexStyle} = [shape = circle, draw = black, 
                           fill=gray!50!white, inner sep = 1pt, outer sep= 0.5pt, minimum size = 2mm, 
                           line width = 1pt]
    \Vertex[a=198 , d=1.5 cm, Math, L=2]{7}
   \Vertex[a=-234 , d=1.5 cm, Math, L=2]{9}

\tikzstyle{VertexStyle}=[draw=none]
   \Vertex[a=54 , d=2.1 cm, Math, L=v_1]{v1}
   \Vertex[a=18 , d=2.1 cm, Math, L=v_2]{v2}
   \Vertex[a=-18 , d=2.1 cm, Math, L=v_3]{v3}
   \Vertex[a=-54 , d=2.1 cm, Math, L=v_4]{v4}
   \Vertex[a=270 , d=2.1 cm, Math, L=v_5]{v5}
   \Vertex[a=234 , d=2.1 cm, Math, L=v_6]{v6}
   \Vertex[a=198 , d=2.1 cm, Math, L=v_7]{v7}
   \Vertex[a=-198 , d=2.1 cm, Math, L=v_8]{v8}
   \Vertex[a=-234 , d=2.1 cm, Math, L=v_9]{v9}
\end{tikzpicture}

\columnbreak

\begin{tikzpicture}[scale=0.65]
\tikzstyle{VertexStyle} = [shape = circle, draw = black, 
                           fill=white, inner sep = 1pt, outer sep = 0.5pt, minimum size = 2mm, 
                           line width = 1pt]
\SetUpEdge[lw=1.5pt] 
\SetGraphUnit{2} 
\tikzset{EdgeStyle/.style={-}}
 \draw (1.5,0) arc (0:360: 1.5 cm);
    \Vertex[a=90 , d=1.5 cm, Math, L=s]{0}
   \Vertex[a=54 , d=1.5 cm, Math, L=1]{1}
   \Vertex[a=18 , d=1.5 cm, Math, L=1]{2}
   \Vertex[a=-18 , d=1.5 cm, Math, L=1]{3}
   \Vertex[a=-54 , d=1.5 cm, Math, L=1]{4}
  \Vertex[a=270 , d=1.5 cm, Math, L=1]{5}
     \Vertex[a=234 , d=1.5 cm, Math, L=1]{6}
      \Vertex[a=198 , d=1.5 cm, Math, L=0]{7}
   \Vertex[a=-234 , d=1.5 cm, Math, L=0]{9}
\tikzstyle{VertexStyle} = [shape = circle, draw = black, 
                           fill=gray!50!white, inner sep = 1pt, outer sep= 0.5pt, minimum size = 2mm, 
                           line width = 1pt]   

  \Vertex[a=-198 , d=1.5 cm, Math, L=2]{8}
\tikzstyle{VertexStyle}=[draw=none]
   \Vertex[a=54 , d=2.1 cm, Math, L=v_1]{v1}
   \Vertex[a=18 , d=2.1 cm, Math, L=v_2]{v2}
   \Vertex[a=-18 , d=2.1 cm, Math, L=v_3]{v3}
   \Vertex[a=-54 , d=2.1 cm, Math, L=v_4]{v4}
   \Vertex[a=270 , d=2.1 cm, Math, L=v_5]{v5}
   \Vertex[a=234 , d=2.1 cm, Math, L=v_6]{v6}
   \Vertex[a=198 , d=2.1 cm, Math, L=v_7]{v7}
   \Vertex[a=-198 , d=2.1 cm, Math, L=v_8]{v8}
   \Vertex[a=-234 , d=2.1 cm, Math, L=v_9]{v9}
 \draw [->] (-5.5,0)--(-3,0);
\end{tikzpicture}

\columnbreak

\begin{tikzpicture}[scale=0.65]
\tikzstyle{VertexStyle} = [shape = circle, draw = black, 
                           fill=white, inner sep = 1pt, outer sep = 0.5pt, minimum size = 2mm, 
                           line width = 1pt]
\SetUpEdge[lw=1.5pt] 
\SetGraphUnit{2} 
\tikzset{EdgeStyle/.style={-}}
 \draw (1.5,0) arc (0:360: 1.5 cm);
    \Vertex[a=90 , d=1.5 cm, Math, L=s]{0}
   \Vertex[a=54 , d=1.5 cm, Math, L=1]{1}
   \Vertex[a=18 , d=1.5 cm, Math, L=1]{2}
   \Vertex[a=-18 , d=1.5 cm, Math, L=1]{3}
   \Vertex[a=-54 , d=1.5 cm, Math, L=1]{4}
  \Vertex[a=270 , d=1.5 cm, Math, L=1]{5}
   \Vertex[a=198 , d=1.5 cm, Math, L=1]{7}
   \Vertex[a=234 , d=1.5 cm, Math, L=1]{6}
  \Vertex[a=-198 , d=1.5 cm, Math, L=0]{8}
   \Vertex[a=-234 , d=1.5 cm, Math, L=1]{9}

\tikzstyle{VertexStyle}=[draw=none]
   \Vertex[a=54 , d=2.1 cm, Math, L=v_1]{v1}
   \Vertex[a=18 , d=2.1 cm, Math, L=v_2]{v2}
   \Vertex[a=-18 , d=2.1 cm, Math, L=v_3]{v3}
   \Vertex[a=-54 , d=2.1 cm, Math, L=v_4]{v4}
   \Vertex[a=270 , d=2.1 cm, Math, L=v_5]{v5}
   \Vertex[a=234 , d=2.1 cm, Math, L=v_6]{v6}
   \Vertex[a=198 , d=2.1 cm, Math, L=v_7]{v7}
   \Vertex[a=-198 , d=2.1 cm, Math, L=v_8]{v8}
   \Vertex[a=-234 , d=2.1 cm, Math, L=v_9]{v9}
 \draw [->] (-5.5,0)--(-3,0);

\end{tikzpicture}
\end{multicols}
\caption{The principal avalanche of $1^301^5 + 1_{v_{6}}$.} % The
                                % shaded vertices are unstable.} 
\label{1p01q}
\end{figure}

Based on what we've seen in Example \ref{ex:cycle:2}, we may guess
that the toppling monomials associated to $1^{p-1}01^{n-p}$ are
related to the toppling monomials of $C_{p}$ and $C_{n-p+1}$.  This
turns out to be true and will now formally state and prove this
idea. First, we must introduce some useful notation.

\begin{defn}
\label{def:ReindexNotation}
Let $q$ be an integer. Then  $C_{n+1}^{q}$ will denote the cycle graph
on $n+1$ vertices labeled 
$v_{q+1}, \dots, v_{q+n}, s$.
%Also, $C_2$ will denote the graph with 2 vertices and 2 edges between
%them. We will consider this graph as a 2-cycle.
\end{defn}

%\begin{rmk}
%\label{rmk:for2}
%Let $c = 1^{p-1}01^{n-p}$. 
%Note that this theorem does not hold for the principal avalanche of
%$c_1 = 101^{n-2}$ at $v_1$ or the principal avalanche of $c_2 =
%1^{n-2}01$ at $v_n$. When sand is added to $v_1$ in $c_1$ only $v_1$
%will topple since a grain of sand will go to the sink and to $v_2$,
%which will remain stable. Similarly, when sand is added to $v_n$ in
%$c_2$ only $v_n$ topples. 
%\end{rmk}

\begin{thm}
\label{thm:CycleNonMaxRecurrents} 
Let $b_p =1^{p-1}01^{n-p}$ be a
recurrent on $C_{n+1}$ such that $1 \leq p \leq n$.

\begin{enumerate}
\item[(1)] 
If $1 \leq i \leq p-1$, then 
$\mu_{C_{n+1}} (b_p,v_{i}) = \mu_{C_{p}}(1^{p-1}, v_i)$,

\item[(2)] 
If $p+1 \leq i \leq n$, then
$\mu_{C_{n+1}} (b_p,v_{i}) = \mu_{C_{n-p+1}^{p}} (1^{n-p},v_{i})$,

\item[(3)]
$\mu_{C_{n+1}}(b_p,v_p) = 1$.
\end{enumerate}
\end{thm}

\begin{proof} 
%Since $b_p = 1^{p-1}01^{n-p}$, it is clear that $b_p+1_{v_p}$ is stable
%and $\mu_{C_{n+1}}(b_p,v_p) = 1$.
%
Let $1 \leq i \leq p-1$, this implies $p\geq 2$. If $p=2$, then
$i=1$, and $b_p + 1_{v_1} = 201^{n-2} \leadsto 0^21^{n-2}$. So 
$\mu_{C_{n+1}}(b_p,v_{1}) = x_{1} = \mu_{C_{2}}(1, v_{1})$ and the
result holds.

Assume $p\geq 3$. If $i = 1$, then $v_{1},\dots, v_{p-1}$ is a legal
toppling sequence and $b_{p} + 1_{v_{1}} \leadsto
1^{p-2}01^{n-p+1}$. Lemma \ref{lem:SinkVertices} implies
$\mu_{C_{n+1}}(b_{p},v_{1}) = x_{1}x_{2}\dotsm x_{p-1} =
\mu_{C_{p}}(1^{p-1},v_{1})$. Similarly, if $i = p-1$, then
$v_{p-1},v_{p-2},\dots, v_{1}$ is a legal toppling sequence and 
$b_{p} + 1_{v_{p-1}} \leadsto 01^{n-1}$. So again 
%Lemma \ref{lem:SinkVertices} implies 
$\mu_{C_{n+1}}(b_{p},v_{p-1}) =
x_{1}x_{2}\dotsm x_{p-1} =  \mu_{C_{p}}(1^{p-1},v_{p-1})$.

If  $p\geq 3$ and $2 \leq i \leq p-2$, we have
$b_p + 1_{v_i} = 1^{p-1}01^{n-p} + 1_{v_{i}} =
1^{i-1}21^{p-i-1}01^{n-p}$. Note that $v_i, v_{i-1},\dots, v_1, v_{i+1},
\dots, v_{p-1}$ is a legal toppling sequence. So  $b_p + 1_{v_i} \leadsto
01^{i-2}21^{p-i-2}01^{n-p+1}$. From this computation we can deduce two
things. First, during the stabilization of $b_{p} + 1_{v_{i}}$, none
of the vertices $v_{p},\dots, v_{n}$ will topple. Second, the
sandpile $01^{i-2}21^{p-i-2}01^{n-p+1}$ is in fact the sandpile
$c_{1}$ defined in Lemma \ref{lemma:OtherStages} when $n+1 = p$ with
the string $1^{n-p+1}$ concatenated at the end. Therefore, from these
two observations we conclude that the
principal avalanche resulting from  $b_{p} + 1_{v_{i}}$ in $C_{n+1}$
follows the  pattern described in Lemma \ref{lemma:OtherStages}
for $C_{p}$. Hence $\mu_{C_{n+1}} (b_p,v_{i}) = \mu_{C_{p}}(1^{p-1}, v_i)$.

A similar argument works for the case $p+1 \leq i \leq n$ with the
exception that now the vertices that topple are $v_{p+1},\dots,
v_{n}$. So $\mu_{C_{n+1}} (b_p,v_{i}) = \mu_{C_{n-p+1}^{p}}
(1^{n-p},v_{i})$.
Finally, it is clear that $b_{p} + 1_{v_{p}}$ is stable. So 
$\mu_{c_{n+1}}(1^{p-1}01^{n-p},v_{p}) = 1$.
\end{proof}

%Now we know how to describe the principal avalanches for all
%recurrents of a cycle. 
The following result follows immediately from Theorem
\ref{thm:CycleTopplingMonomialsMax} and Theorem
\ref{thm:CycleNonMaxRecurrents}.  

\begin{cor}
\label{thm:Cycles}
The avalanche polynomial for $C_{n+1}$ for $n\geq 1$ is 
$$\sum_{i=1}^n \mu_{C_{n+1}}(1^{n}, v_i) 
+ \sum_{p=2}^{n} \sum_{i=1}^{p-1} \mu_{C_{p}}(1^{p-1}, v_i)
+ \sum_{p=1}^{n-1} \sum_{i=p+1}^{n} \mu_{C_{n-p+1}^{p}}(1^{n-p}, v_{i}) 
%+ x_1 + x_n 
+ n,$$
where $\displaystyle \mu_{C_{q+1}}(1^{q}, v_i) = \prod_{j=1}^m
x_j\dotsm x_{q-j+1}$, $1\leq i\leq q$, and $m = \min\{i,q-i+1\}$.
\end{cor}

%%%%%%%%  COMPLETE GRAPHS %%%%%%%%%%%%%%

\section{Avalanche Polynomials of Complete Graphs}\label{complete}

In this section we will compute the avalanche polynomial of the
complete graph $K_{n+1}$ on $n+1$ vertices $v_1,v_{2},\dots,
v_{n},s$.  In Example \label{ex:CycleToppling}  we computed
$\sA_{K_{3}}(x_{1},x_{2})$. Using SageMath \cite{sage}, we can compute
the avalanche polynomial of $K_{4}$:

%\begin{ex}
%\label{ex:K4}
%\end{ex}

$$\sA_{K_4}(x_1,x_2,x_3)=9x_1x_2x_3+2x_1x_2+2x_1x_3+2x_2x_3+3x_1+3x_2+3x_3+24.$$ 

Note that $\sA_{K_{4}}(x) = 9x^3+6x^2+9x+24$. So the set of principal
avalanches of size $m$ is evenly partitioned into $\binom{3}{m}$
subsets for $0\leq m \leq 3$. Moreover, $\sA_{K_4}(x_1,x_2,x_3)$ is a
linear combination of elementary symmetric polynomials.
We will show that this characterizes the avalanche polynomial of $K_{n+1}$.

%The elementary symmetric polynomial allows us to write down the
%avalanche polynomial of $K_n$.

\begin{defn}
Let $m$ be an integer such that $0 \leq m \leq n$.
The \emph{elementary symmetric polynomial} of degree $m$ on variables
$x_1,x_2,\dots,x_n$ is 
$$e_{m}(x_{1},\dots, x_{n}) = \sum_{\substack{A \subseteq [n] \\ |A| =
    m}} \prod_{i \in A} x_{i}.$$ 
\end{defn}
 
Observe that $e_{0}(x_{1},\dots,x_{n}) = 1$ and the number of terms in
$e_{m}(x_{1},\dots, x_{n})$ is $\binom{n}{m}$.

As mentioned in Section \ref{background},  
the number of recurrent sandpiles of a graph $G$ equals the number of
spanning trees of $G$.  Cayley's formula implies that $K_{n}$ has 
$n^{n-2}$ recurrent sandpiles. In order to study the principal
avalanches in this graph, we will use a beautiful result first proved
in \cite{CR00} that establishes a bijection between recurrent
sandpiles in $K_{n+1}$ and $n$-parking functions. 

%\subsection{Recurrents and Parking Functions}

\begin{defn}
Given a function $p:\{0,1,\dots,n-1\} \to \{0,1,\dots,n-1\}$, let
$a_0\leq a_1\leq\cdots\leq a_{n-1}$
be the non-decreasing rearrangement of $p(0),\dots, p(n-1)$.  We say
that $p$ is an \emph{$n$-parking function} provided that $a_i\leq i$
for $0 \leq i \leq n-1$.   
\end{defn}

Note that the parking function $p$ can be represented by the vector 
$(p(0),p(1),\dots,p(n-1))$.

\begin{prop}[{\cite[Proposition 2.8]{CR00}}]
\label{prop:park}
The sandpile $c$ is recurrent on $K_{n+1}$ if and only if
$\smax_{K_{n+1}} - c $ is an $n$-parking function. 
\end{prop}

 It is clear from the definition that any permutation of a parking
 function is also a parking function.  We can concatenate parking
 functions to obtain new parking functions.

\begin{lemma}
\label{lemma:PFConcatenate}
Let $p=(p_0,p_1,\dots,p_{m-1})$ and $q=(q_0,q_1,\dots,q_{n-1})$ be two
parking functions.  Then
$(p_0,p_1,\dots,p_{m-1},q_0+m,q_1+m,\dots,q_{n-1}+m)$ is also a
parking function. 
\end{lemma}

\begin{proof}
Let $a_0 \leq a_1 \leq \dotsm \leq a_{m-1}$ and $b_0 \leq b_1 \leq
\dotsm \leq b_{n-1}$ be non-decreasing rearrangements of $p$ and $q$,
respectively. Note 
$b_0+m \leq b_1+m \leq \dotsm \leq b_{n-1}+m$ and $b_i + m \leq i +
m$ for each $i=0,\dots, n-1$, since $q$ is a parking function. 
Moreover, $a_{m-1}\leq m-1 < m = b_0 + m$.  So 
$$a_0 \leq a_{1} \leq \cdots \leq a_{m-1} < b_0 + m\leq b_1 + m \leq
\cdots \leq b_{n-1} + m$$ 
and each term is less than or equal to its index.  
%Therefore, the above inequality is the non-decreasing rearrangement of
%a parking function. Since any permutation of a parking function is
%another parking function,
%Thus $(p_0,p_1,\dots,p_{x-1},q_0+x,q_1+x,\dots,q_{y-1}+x)$
%is a parking function. 
\end{proof}

\subsection{The Avalanche Polynomial of $K_{n+1}$}

The following lemma gives a partial description of a sandpile $c$
given the size of a principal avalanche. 

\begin{lemma}
\label{lem:WhatTopples}
Let $c$ be a sandpile on $K_{n+1}$. 
Suppose that the principal avalanche resulting from stabilizing
$c+1_{v_{k}}$ has length $m\geq 1$. Let $w_0,w_1,\dots,w_{m-1}$ be the
associated toppling sequence and let $\{u_0,\dots,
u_{n-m-1}\}$ be the set of vertices that do not topple. Assume further,
perhaps after relabeling, that  
$c(u_0)\geq c(u_1)\geq \cdots \geq c(u_{n-m-1})$.  The following are
true: 
\begin{enumerate} 
\item $c(w_0) = c(v_k) = n-1$.
\item $n-i\leq c(w_i)\leq n-1$, for $i = 1,\dots, m-1$.
\item $n-m-i-1\leq c(u_i)\leq n-m-1$, for $i = 0,\dots, n-m-1$.
\end{enumerate}
\end{lemma}

\begin{proof}
First note that since $c$ is stable then $c(v) \leq n-1$.
Since $m \geq 1$, then $w_0 = v_{k}$ must topple. Thus, $c(w_0)=n-1$.
Corollary \ref{lem:ToppleOnce} implies that each $w_{j}$ appears exactly
once in the toppling sequence. Moreover, when a vertex topples, it adds
one grain of sand to every other non-sink vertex of $K_{n+1}$. Thus,
toppling vertices $w_0,\dots, w_{i-1}$ adds $i$ grains of sand to
$w_{i}$. Since this vertex must topple next, then $c(w_i)+i \geq n $. So,
$n-i \leq c(w_i) \leq n-1.$  
On the other hand, since $u_{i}$ does not topple, then
$c(u_i) \leq n-m-1$ for $0\leq i \leq n-m-1$. By Proposition
\ref{prop:park}, $p=\smax_{K_{n+1}} - c$ is an $n$-parking
function. Let $p'$ be its non-decreasing rearrangement.
Note that the first $m$ entries in $p'$ correspond to the $m$
vertices that topple and the last $n-m$ entries correspond to the
vertices that do not topple. So $p'(m+i)= n-1 - c(u_i)$. Since
$p'(m+i) \leq m+i$, then $c(u_i) 
\geq n-m-i$. So, $n-m-i-1 \leq c(u_i) \leq n-m-1$.   
\end{proof}

%Note that in Example \ref{ex:K4} each summand of $\sT_{K_{4}}$ is a
%multiple a summand of $\sigma_{4,k}$ for some $k$ with $0\leq k\leq
%4$.   Now we will show that the summands that appear in
%$\sT_{K_{n+1}}$ are exactly multiples of $\sigma_{n,k}$ for $1\leq k
%\leq n$. 

\begin{prop}
\label{lem:SymmPoly}
Let $\lambda_{m}$ denote the number of principal avalanches of size $m$
in $K_{n+1}$. Then  
$$\sA_{K_{n+1}}(x_{1},\dots,x_{n}) = \sum_{m=0}^n
\frac{\lambda_{m}}{\binom{n}{m}}e_{m}(x_{1},\dots,x_{n}).$$ 
\end{prop}

\begin{proof}
First note that each vertex can topple at most once in any principal
avalanche by Corollary \ref{lem:ToppleOnce}.  So every monomial in
$\sA_{K_{n+1}}$ is square-free and completely characterized by its support.
Fix an integer $m$ with $1\leq m\leq n$. Let $A \subseteq [n]$ with
$|A| = m$ and let $\mu_{A} = \prod_{i \in A} x_i$. Consider the sandpile
$c$ defined by 
$$c(v_i) = \begin{cases}
n-1 & \text{if } i\in A,\\
n-1-m  & \text{if } i\notin A.
\end{cases}$$
 Note that $\smax - c(v_{i}) = 0$ if $i\in A$ and $\smax-c(v_{i}) = m$
 if $i \notin A$.
The non-decreasing rearrangement of $\smax-c$ is is
$(\underbrace{0,0,\dots,0}_m,\underbrace{m,m,\dots,m}_{n-m})$ which is a
parking function.  Thus, by Proposition \ref{prop:park}, the sandpile $c$ is
recurrent.  By Lemma \ref{lem:WhatTopples}, $\mu(c,v_{i}) = \mu_{A}$ for any
vertex $v_{i}$ with $i\in A$. 
To complete the proof, we must show that if $A = (i_{1},\dots,i_{m})$
and $A' = (i_{1}',\dots, i_{m}')$ are two ordered 
subsets of $[n]$ of cardinality $m$, then the number of principal
avalanches that produce the toppling sequence $(v_{i_{1}},\dots, v_{i_{m}})$
equals the number of principal avalanches that produce the toppling
sequence $(v_{i_{1}'},\dots, v_{i_{m}'})$.  However, this follows directly from
the symmetry of $K_{n+1}$. Explicitly, if $\pi$ denotes the permutation that sends
$v_{i_{j}} \to v_{i_{j}'}$ for each $j=1,\dots, m$. Then $\mu_{K_{n+1}}(c,v_{i_{1}}) =
\mu_{A}$ if and only if $\mu_{K_{n+1}}(c',v_{i_{1}'}) = \mu_{A'}$, where
$c'$ is the sandpile obtained by permuting the entries of $c$
according to $\pi$. The explicit form of the coefficient of
$e_{m}(x_{1},\dots,x_{n})$ follows from the fact that this polynomial
has $\binom{n}{m}$ monomials.
\end{proof}

The coefficients $\lambda_{m}$ in Proposition \ref{lem:SymmPoly} were
computed in \cite[Propositions
4 and 5]{CDR04}. Explicitly,  $\lambda_{0} = n(n-1)(n+1)^{n-2}$ and
\[\lambda_{m} = \binom{n}{m}m^{m-1}(n-m+1)^{n-m-1}, \text{ for } 1\leq
m \leq n.\] 
We include a proof of the latter result in order to
correct a mistake in their original argument. However, we also want
to point out that the coefficient $\lambda_{m}$ is also the number of
principal avalanches with burst size $m$ since every non-sink vertex
in $K_{n+1}$ is adjacent to the sink.

\begin{defn}
\label{defn:phi}
Let $c\in \sg(K_{n+1})$ and $v_i\in \rV$ such that
when a grain of sand is added to $v_i$, an avalanche of size $m\geq 1$
occurs. Define the function  
$$\phi: \sg(K_{n+1}) \times \rV \longrightarrow
\rV\times
\binom{\rV\setminus\{v_{i}\}}{m-1}\times
\sg(K_{m})\times\sg(K_{n-m+1}),$$ 
 such that $\phi(c,v_i) = (v_i,J,c_1,c_2)$,  where
$J = \{w_{1},\dots, w_{m-1}\}$ is the set of $m-1$ vertices that
topple  other than $w_{0} = v_{i}$. The sandpile $c_{1}$ in
$K_{m}$ is defined by $$c_{1} = (c(w_{1})-(n-m+1),\dots,
c(w_{m-1})-(n-m+1)),$$ and the sandpile $c_{2}$ in $K_{n-m+1}$ is
defined by the values $c(v_{k})$ for $v_{k}\notin J\cup\{v_{i}\}$.  
\end{defn}

%Note that counting pairs $(c,v_i)$ where the principal avalanche
%started on $c$ at $v_i$ has size $m$ counts the total number of
%principal avalanches of size $m$. We will prove that $\phi$ is a
%bijection and then count the quadruples $(v_i,J,c_1,c_2)$
%instead. First, an example. 

\begin{ex}
\label{ex:phi}
Consider the recurrent sandpile $c = ( 8,7,8,1,0,3,7,2,4 )$ on
$K_{10}$. Note that adding a grain of sand at $v_1$ causes an
avalanche of size $m=4$.  In this case  $J = \{ v_2, v_3, v_7
\}$, $c_1 = ( 7-6,8-6,7-6 ) = (1,2,1)$ and $c_2 = (1,0,3,2,4)$. In
\cite{CDR04}, the authors define the sandpile $c_{1}$ by substracting
$m-2$ instead of $n-m+1$. In here, this would result in the sandpile
$(5, 6, 5)$ that is not even a stable sandpile on $K_{4}$. 
\end{ex}

\begin{lemma}
The map $\phi$ described in Definition \ref{defn:phi} is a bijection.
\end{lemma}

\begin{proof}
First we need to show that the map $\phi$ above is well-defined, that
is, we need to show that $c_1$ and $c_2$ are, in fact, recurrent
sandpiles on $K_m$ and $K_{n-m+1}$, respectively.  
To show $c_1$ is recurrent, let $J = \{w_1,\dots, w_{m-1}\}$ such that
$c(w_i)\leq c(w_{i+1})$ for $1\leq i \leq m-2$.  By Lemma
\ref{lem:WhatTopples}, for $1\leq i \leq
m-1$, we have that  $n-i\leq c(w_i)\leq n-1$. So,  $$n-i - (n-m+1) \leq
c(w_{i}) - (n-m+1) \leq n-1 - (n-m+1)$$ and 
$m-i-1\leq c_1(w_i) \leq m-2.$
This implies $c_1$ is a stable sandpile on $K_m$.
Consider $p_1 = \smax_{K_m} - c_1$.  For $1\leq i \leq m-1$,  
$0 \leq p_1(w_i) \leq i-1$, so $p_1$ is a parking function and $c_1$
is recurrent.
%(noting that we need to reindex so our indices start at $0$). 
%
To show $c_2$ is recurrent, let $\{u_0,\dots, u_{n-m-1}\} =
\rV\setminus (J\cup \{v_i\})$ such that $c(u_i)\leq
c(u_{i+1})$ for $0\leq i \leq n-m-2$.  By Lemma \ref{lem:WhatTopples},
%since this is the set of vertices that do not topple, 
$$n-m-i-1 \leq c(u_i)\leq n-m-1.$$ 
Since $c_2(u_i) = c(u_i)$ then $c_{2}$ is stable in $K_{n-m+1}$.
Consider $p_2 = \smax_{K_{n-m+1}} - c_2$.   For $0\leq i \leq n-m-2$, 
we have $ 0\leq p_2(i) \leq i. $
Since $p_2$ is a parking function, $c_2$ is recurrent.

% First we show that $\phi$ is one-to-one.  Let $a_1,a_2\in
% \sg(K_{n+1})$ and $v_1,v_2\in \widetilde{V}(K_{n+1})$ such that the
% principal avalanche of $a_i$ at $v_i$ for $i=1,2$ has size $m$.
% Suppose $\phi(a_1,v_1) = \phi(a_2,v_2) = (v,J,c_1,c_2)$. Clearly, $v_1
% = v_2 = v$. Also, the $m$ vertices that topple in the principal
% avalanche of $a_i$ at $v_i$ are exactly those vertices in $J\cup v$
% for $i=1,2$. For a vertex $v$ in $J$, $a_1(v) = c_1(v) + n-m+1 =
% a_2(v)$.  Thus, $a_1$ and $a_2$ agree on vertices in $J$.  Let $B$ be
% the set of all the vertices that do not topple, which is the same for
% the principal avalanche of $a_1$ at $v_1$ as for the principal
% avalanche of $a_2$ at $v_2$. For $v\in B$, $a_1(v) = c_2(v) =
% a_2(v)$. This means for each $v \in\widetilde{V}( K_{n+1})$, $a_1(v_i)
% = a_2(v_i)$. Therefore $(a_1,v_1) = (a_2,v_2)$. 

% Let $\widetilde{V}(K_m) = \{u_1,\dots, u_{m-1}\}$ and
% $\widetilde{V}(K_{n-m+1}) = \{w_1,\dots, w_{n-m}\}$.  
% We will construct a sandpile $c\in \sg(K_{n+1})$ such that the
% principal avalanche of $c$ at $v$ has size $m>0$ and $\phi(c,v) =
% (v,J,c_1,c_2)$.  
% Define $c(v) = n-1$.  
% Let $J = \{j_1,\dots,j_{m-1}\}\subset \widetilde{V}(K_m)\setminus
% \{v\}$ and define $c(j_i)$ to be $c_1(v_i) + n-m+1$.  
% Let $\mathcal{U}=\{u_1,\dots,u_{n-m}\} =
% \widetilde{V}(K_{n+1})\setminus (J\cup\{v\})$.  Let $c(u_i) =
% c_2(w_i)$.  Note that $\phi(c,v) = (v,J,c_1,c_2)$. 

The fact that the map $\phi$ is injective follows immediately from the
definition of $c_{1}$ and $c_{2}$.  
Finally, we will show that $\phi$ is onto. Given $(v,J,c_1,c_2)$ 
we define $c$ as follows. First, let $c(v) = n-1$. Now, for each
$w_{i} \in J$, define $c(w_{i})$ by adding $n-m+1$ to the $i$th entry
in $c_{1}$. The remaining $n-m$ entries in $c$ are filled with the
entries in $c_{2}$. 
Since $c_1$ and $c_2$ are recurrent, Proposition \ref{prop:park} implies
$p_{1} = \smax_{K_m}-c_1$ and $p_{2} = \smax_{K_{n-m+1}}-c_2$  are
$(m-1)$ and $(n-m)$-parking functions, respectively.
By Lemma \ref{lemma:PFConcatenate}, concatenating $p_1$ and $p_2
+ \overline{m}$ defines an $(n-1)$-parking function
$p'$, where $\overline{m} = (m,\dots, m)$. Furthermore, concatenating $0$ and $p'$
gives an $n$-parking function $p$. 
Moreover, $\smax_{K_{n+1}}-p$ is a rearrangement
of $c$, so $c$ is a recurrent sandpile on $K_{n+1}$. 
Clearly,  $\phi(c,v) = (v,J,c_1,c_2)$ and this completes the proof. 
\end{proof}
 
From the bijection $\phi$ we are able to compute the number
$\lambda_{m}$ of
principal avalanches of size $m>0$. Given  $A \subseteq [n]$ with
$|A| = m$, Proposition \ref{lem:SymmPoly} states that $\lambda_{m}/\binom{n}{m}$ is the
number of principal avalanches with avalanche monomial $\mu_{A} =
\prod_{i \in A} x_i$. The bijection $\phi$ implies that this number
equals the number of four-tuples $(v_i,J,c_1,c_2)$ with $J\cup
\{v_{i}\} = A$.  There are $m$ ways to pick $v_i$.
Cayley's formula implies that the number of recurrents on $K_m$ and
$K_{n-m+1}$ is $m^{m-2}$ and $(n-m+1)^{n-m-1}$,
respectively. Therefore, 
\[\lambda_{m} = \binom{n}{m}\cdot m\cdot m^{m-2}(n-m+1)^{n-m-1}= \binom{n}{m} m^{m-1}(n-m+1)^{n-m-1}.\] 

% \begin{thm}
% \label{thm:complete} The toppling polynomial of $K_{n+1}$ is 
% $$\sT(K_{n+1}) = n(n-1)(n+1)^{n-2} + \sum_{m=1}^n
% m^{m-1}(n-m+1)^{n-m-1}\sigma_{n,m}.$$ 
% \end{thm}

% \begin{proof}
% By \cite{CDR04} we know that the number of principal avalanches of
% size 0 is $n(n-1)(n+1)^{n-2}$. By Lemma \ref{lem:SymmPoly}, any other
% term looks like $\lambda\cdot\gamma$ for $\lambda \in\mathbb{Z}$,
% $\lambda\geq 0$ and $\gamma$ a term of  $\sigma_{n,m}$ some $1\leq m
% \leq n$.  Suppose $\gamma = \prod_{i \in A} x_i$ for some $A \subseteq
% [n]$ and $|A|=m$.  
% We will count all of the four-tuples $(v_k,J,c_1,c_2)$ such that $J
% \cup v_k=\{v_i: i\in A\}$ .  
% There are $m$ ways to pick $v_k$ because this vertex must be an
% element in $\{v_i: i\in A\}$ which has size $m$.  
% After $v_k$ is chosen, $J$ is set to the other $m-1$ vertices with
% indices in $A$.  
% The number of recurrents of $K_m$ and $K_{n-m+1}$ is $m^{m-2}$ and
% $(n-m+1)^{n-m-1}$, respectively. Therefore, $m\cdot
% m^{m-2}(n-m+1)^{n-m-1}=m^{m-1}(n-m+1)^{n-m-1}$ is the coefficient for
% $\gamma$. 
% \end{proof}

%%%%%%%%  WHEEL GRAPHS %%%%%%%%%%%%%%

\section{The Avalanche Polynomial of the Wheel}\label{wheel}

The wheel graph, denoted $W_{n}$, is a cycle on $n \geq 3$ vertices
with an 
additional \emph{dominating vertex}.  Throughout, the vertices in the
cycle will be labeled clockwise as $v_0,\dots, v_{n-1}$, where the
indices are taken modulo $n$. 
The
dominating vertex, denoted $s$, will always be assumed to be the sink.

The sandpile group of $W_{n}$ was first computed by Biggs in
\cite{B99}:
\[
\sg(W_{n}) =
\begin{cases}
 \zz_{l_{n}} \oplus \zz_{l_{n}} & \text{ if } n \text{ is
  odd }\\
 \zz_{f_{n}} \oplus \zz_{5f_{n}} & \text{ if } n \text{ is
  even }
\end{cases}
\]
where $\{l_{n}\}$ is the Lucas sequence and $\{f_{n}\}$ is the
Fibonnaci sequence. These sequences are defined by initial
conditions $l_{0} = 2, l_{1} =1$ and $f_{0} = 1, f_{1} = 1$,
respectively, and the recursion $x_{n} = x_{n-1} + x_{n-2}$. There are
many relationships among these numbers. For example, $l_{n} = f_{n-1}
+ f_{n+1}$. Morever, the order of $\sg(W_{n})$
equals the number of spanning trees $\tau(W_{n})$ in $W_{n}$. This 
number equals $\tau(W_{n}) = l_{2n} - 2$, see \cite{H72, BY06}. 

We have already computed the avalanche polynomial of $W_{3}$
since $W_{3} = K_{4}$. In this case, 
\[\sA_{W_{3}}(x_{0},x_{1},x_{2}) = 9x_{0}x_{1}x_{2} + 2(x_{0}x_{1} +
x_{1}x_{2} + x_{2}x_{0}) + 3(x_{0}+x_{1}+x_{2}) + 24.\]
Observe that the set of principal avalanches of size $0 < m < 3$ is
evenly partitioned into $n=3$ subsets. Moreover,
$\sA_{W_3}(x_0,x_1,x_2)$ is a linear combination of \emph{cyclic
  polynomials}. We will show that this characterizes
$\sA_{W_n}(x_0,\dots, x_{n-1})$. 

\begin{defn}
Let $m$ be an integer such that $1 \leq m \leq n-1$. We will denote by 
$w_{m}(x_{0},\dots,x_{n-1})$ the 
\emph{cyclic polynomial} of degree $m$ on variables
$x_0,\dots, x_{n-1}$ defined as
\[w_{m}(x_{0},\dots,x_{n-1}) = \sum_{i=0}^{n-1}x_{i}x_{i+1}\dotsm x_{i+m-1}.\]
\end{defn}

First note that for each $1\leq m \leq
n-1$, $w_{m}$ is the sum of $n$ terms of degree $m$. 
For example, $w_{1} = x_{0} + \dotsm + x_{n-1}$.
For the case $m=n$, the above definition would give
\[w_{n} = \sum_{i=0}^{n-1}x_{i}x_{i+1}\dotsm x_{i+n-1} = nx_{0}\dotsm x_{n-1}. \]
However, we will remove the coefficient $n$ and define 
\[w_{n}(x_{0},\dots,x_{n-1}) = x_{0}\dotsm x_{n-1}.\]

In \cite{DaRoFPSAC03}, Dartois and Rossin gave exact results on
the distribution of avalanches on $W_{n}$. Their approach
consisted in showing that the recurrents on $W_{n}$ can be
seen as words of a \emph{regular language}. They built an automaton
associated to this language and used the concept of \emph{transducers}
to determine the exact distribution of avalanche lengths in this
graph. Here we take a different approach focused solely on
the structure of the recurrent sandpiles.

Note that the degree of every non-sink vertex in $W_{n}$ is $3$. So
any stable sandpile on this graph can be written as a word of length $n$
in the alphabet $\{0,1,2\}$. 
Applying Dhar's Burning Criterion (Proposition \ref{prop:Burn}),
Cori and Rossin \cite{CR00} showed that a sandpile on $W_{n}$ is
recurrent if and only if there is at least one vertex 
with 2 grains of sand and between any two vertices with $0$ grains,
there is at least one vertex with 2 grains. 
% In order to compute
%the avalanche polynomial of $W_{n}$, we will first partition the set
%of recurrent sandpiles into subsets whose elements are sandpiles with
%equal \emph{maximal} avalanche size. Then we will consider a map from
%each of these subsets into the recurrent sandpiles of a \emph{fan graph}.

\begin{defn}
Let $m$ be an integer with $1\leq m \leq n-1$. 
A sandpile $c$ in $W_{n}$ has a \emph{maximal 2-string of length $m$}
if there 
are vertices $v_i,v_{i+1},\dots, v_{i+m-1}$, such that  $c(v_i) =
\cdots = c(v_{i+m-1}) = 2$ and $c(v_{i-1}) \neq 2 \neq c(v_{i+m})$.  
\end{defn}

Note that $\smax_{W_{n}} = 2^{n}$ is the unique recurrent
with a maximal $2$-string of length $n$.

\begin{lemma}\label{lem:Toppling2strings}
Let  $c\in \sg(W_{n})$. The principal avalanche of $c$ at a non-sink
vertex $v$ has size $1\leq m \leq n-2$ if and only if $v$ is part of a
maximal $2$-string of length $m$. 
\end{lemma}

\begin{proof}
Suppose a grain is added to a vertex $v$ that is part of a
maximal 
$2$-string of length $m$.  Since $m<n-1$, the two non-sink vertices
adjacent to the ends of the $2$-string are distinct. Thus,
exactly the $m$ vertices in the maximal $2$-string will topple. On the
other hand, if $v$ is part of a longer or
shorter maximal $2$-string, the avalanche will not have size $m$. 
\end{proof}

Lemma \ref{lem:Toppling2strings} implies that for each $1\leq m \leq
n-2$,  
%counting these maximal
%2-strings is equivalent to finding the distribution of principal
%avalanche of sizes. 
the number $\lambda_{m}$ of principal avalanches of size $m$ equals
$m$ times the number of maximal $2$-strings of length $m$ over all
recurrents. 
%in $W_{n}$. 
This is not the case for avalanches of size
$n-1$ or $n$ as the following simple lemma shows.

\begin{lemma}
\label{lem:SizeNAvalanche}
For any non-sink vertex $v$ in $W_{n}$,
$\mu(2^n,v) = x_0\cdots x_{n-1}$. Also, let $p$ be an integer with
$0\leq p \leq n-1$. For any non-sink vertex $v$ with $v\neq v_{p}$ we
have   
\begin{align*}
\mu_{W_{n}}(2^p12^{n-p-1},v) &= x_0\cdots x_{n-1}, \\% \text{ and }\\
\mu_{W_{n}}(2^p02^{n-p-1}, v) &= 
\frac{x_0\cdots x_{n-1}}{x_{p}}.
\end{align*}
This implies $\lambda_{n} = n^{2}$ and $\lambda_{n-1} = n(n-1)$.
\end{lemma}

% \begin{proof}
% Suppose $c = 2^p02^{n-p-1}$ where $p\geq 0$. Let $v\in\widetilde{V}$
% such that $c(v) = 2$.  We know each vertex topples at most once in the
% principal avalanche of $c$ at $v$ by Corollary \ref{lem:ToppleOnce}.
% Note that $v_1,v_2,\dots, v_{p},v_{p+2},\dots, v_n$ is a legal
% toppling sequence.  Applying this toppling sequence $v_{p+1}$ gains 2
% grains of sand, but remains stable as $c(v_{p+1})=0$.  Thus,
% $\mu(2^p02^q, v) = x_1\cdots x_p x_{p+2} \cdots x_n$. 

% Suppose $c = 2^p12^q$.  The avalanche caused by adding a grain of sand
% to $v$ where $c(v) =2$ is similar to above, but when $v_{p+1}$ gains
% two grains of sand it becomes unstable and topples.  Thus every vertex
% topples at least once, but by Corollary \ref{lem:ToppleOnce}, each
% vertex topples at most once.  Thus, $\mu(2^p12^q,v) = x_1\cdots
% x_n$. Similarly, when $c=2^n$, every vertex must topple once, and
% every vertex topples at most once. 
% \end{proof}

\begin{proof}
Clearly the avalanche monomials for the given recurrents satisfy the
above claims.
Note also that the only avalanches of size $n$ occur on
recurrents of the form $2^p12^{n-p-1}$ and $2^n$. So there are
$n(n-1)+n = n^2$ avalanches of size $n$. The avalanches of size $n-1$
occur on recurrents of the form $2^p02^{n-p-1}$. Hence there are
$n(n-1)$ avalanches of size $n-1$.
\end{proof}

For each $1\leq m \leq n-2$, 
we will count the maximal 2-strings of length $m$ by establishing
a map from the set of recurrents on $W_{n}$ with a given maximal
2-string of length $m$ into the set of recurrents on the \emph{fan
graph} $F_{n-m}$. Let $k\geq 2$, 
the fan graph on $k+1$ vertices, denoted $F_{k}$, is a path on $k$
vertices, plus an additional dominating vertex $s$. 
%If $n=1$, then $F_{1}$ will
%denote the graph with 2 vertices and 2 edges between them.  

% The following algorithm gives a bijection between the spanning trees
% of $G$ and the recurrents of $(G,s)$ where $(G, s)$ is
% a sandpile graph: 

%\begin{addmargin}[2em]{2em}
% Let $\pi$ be the sandpile with $\pi_j = \weight(v_j,s)$ for each
% $v_j\in \widetilde{V}(G)$. 
% Order the vertices of $G$ arbitrarily. Let $c\in \sg(G,s)$. Set
% $a=c+\pi$, and let $\tau=\emptyset \subset E$ be the empty tree. For the
% first step, add to $\tau$ all edges connecting the sink to an unstable
% vertex of $a$. Next, let $v$ be the first (with respect to the chosen
% ordering of the vertices) unstable vertex of $a$. Replace $a$ by
% $a-\rD v$, and add edges to $\tau$ connecting $v$ to vertices that are
% unstable and are not on edges already in $\tau$. Repeat: choose the
% first unstable vertex of $a$, topple it, update $a$, add edges to
% $\tau$ from the toppled vertex to any unstable vertices that are not
% already on any edge of $\tau$.  
%\end{addmargin}

\begin{prop}
\label{prop:NonMaxRecurrents}
For each $1\leq m \leq n-2$,
there is a bijection between the set of recurrents on $W_{n}$ with
a maximal 2-string of length $m$ starting at $v_0$ and the set of
recurrents on $F_{n-m}$. 
%Explicitly, $c$ is a recurrent sandpile on
%$W_{n}$ with a maximal $2$-string starting at $v_{0}$ if and only if
%$c(v_{m},\dots, v_{n-1})$ is a recurrent sandpile on $F_{n-m}$.
\end{prop}

\begin{proof}
Let $c$ be a recurrent sandpile on $F_{n-m}$.  
Dhar's Burning Criterion (Proposition \ref{prop:Burn}) implies that
adding 1 grain of sand to each vertex must result in an avalanche
where every vertex topples exactly once. This 
implies that at least one of the endpoint vertices has 1 grain of sand
or both endpoints have 0 grains of sand and there is an internal
vertex with 2 grains of sand. Moreover, if a 
vertex has $0$ grains of sand, then its neighbors must topple before
it, hence there are no consecutive vertices with $0$ grains of
sand. For the same reason, between any two vertices with $0$ grains of
sand there cannot be a sequence of $1$'s. In summary, $c$
is a recurrent on $F_{n-m}$ if and only if  between any two vertices
with $0$ grains of sand there is a vertex with $2$ grains of
sand.
Hence $c$ is a recurrent
sandpile on $F_{n-m}$ if and only if the sandpile obtained by
prepending a string of $m$ 2's to $c$ is recurrent on $W_{n}$. 
%This concludes the proof.
\end{proof}

It is well-known that the number of spanning trees in the fan graph
$F_{k}$ is precisely the Fibonnaci number $f_{2k}$, see \cite{H72}. 
So Proposition \ref{prop:NonMaxRecurrents} implies that for each
$1\leq m \leq n-2$, there are $f_{2(n-m)}$ recurrent sandpiles that
have a maximal $2$-string of length $m$ starting at $v_0$.

\begin{thm}  \label{thm:wheel}
Given $n\geq 3$, the avalanche polynomial of the
  wheel graph $W_{n}$ is
$$\sA_{W_{n}} = n^2w_{n}(x_{0},\dots, x_{n-1}) + \sum_{m = 1}^{n-1} m\cdot f_{2(n-m)} w_{m}(x_{0},\dots, x_{n-1}) + 2n\left(f_{2n-1}-1\right).$$
%where $C$ is $n\cdot \abs{\sg(W_{n+1})} - 2n^2+n -
%\sum_{m=2}^{n-2}n\cdot m \cdot f_{2(n-m)}$. 
\end{thm}

\begin{proof}
In Lemma \ref{lem:SizeNAvalanche} we showed that 
$\lambda_{n} = n^{2}$.
This lemma also shows that 
the avalanches of size $n-1$ are
caused by adding a grain of sand at any vertex with $2$ grains 
in any 
recurrent of the form $2^p02^{n-p-1}$ with $0\leq p \leq n-1$. Since
$\mu_{W_{n}}(2^p02^{n-p-1},v) = x_0\cdots x_{n-1}/x_{p}$,
for any $v\neq v_{p}$. Then
the degree $n-1$ part of %the avalanche polynomial
$\sA_{W_{n}}$ equals $(n-1)w_{n-1}(x_{0},\dots, x_{n-1})$. Note that
when $m = n-1$ we have $f_{2(n-m)} = f_{2} = 1$.

Now let $1\leq m \leq n-2$. 
Proposition \ref{prop:NonMaxRecurrents} implies
that there are $f_{2(n-m)}$ recurrents on $W_{n}$ with
a maximal 2-string of length $m$ starting at $v_0$.
So by Lemma \ref{lem:Toppling2strings},
there are $mf_{2(n-m)}$ principal avalanches
with avalanche monomial $x_{0}\cdots x_{m-1}$. 
This lemma also shows that any avalanche of size $m$ must occur at a
maximal $2$-string of length $m$. So the only possible avalanche
monomials of degree $m$ are the monomials occuring in the cyclic
polynomial $w_{m}$. Moreover, the cyclic symmetry of $W_{n}$ implies
that the number of principal avalanches that produce the toppling
sequence $(v_{0},v_{1},\dots, v_{m-1})$ equals the number of
principal avalanches that produce the toppling sequence
$(v_{i},v_{i+1},\dots, v_{i+m-1})$ for any $0\leq i\leq
n-1$. Therefore, for any $1\leq m \leq n-2$, the degree $m$ part of  
$\sA_{W_{n}}$ equals $mf_{2(n-m)}w_{m}(x_{0},\dots, x_{n-1})$.

%Rotating the recurrent and the vertex at which we add sand, we get
%$m\cdot f_{2(n-m)}$ principal avalanches of $c$ at $v$ that have 
%$$\mu(c,v) = x_ix_{i+1}\cdots x_{i+m-2}$$

Lastly, note that an avalanche of size $0$ is produced by adding a
grain of sand to a vertex with $0$ or $1$ grains of sand. So
$\lambda_{0}$ equals the number of $0$'s and $1$'s in every recurrent.
%in $W_{n}$. 
Since there are $l_{2n} - 2$ recurrents on $W_{n}$, then
$\lambda_{0}$ equals $n(l_{2n} - 2)$ minus the total number of $2$'s
in every recurrent. 
%Lemma \ref{lem:Toppling2strings} 
Recall that for $1\leq m \leq n-2$, the number $\lambda_{m}$ of
principal avalanches of size $m$ equals $m$ times the number of maximal
$2$-strings of length $m$ over all recurrents, that is, $\lambda_{m}$
equals the total number of $2$'s in every  maximal
$2$-string of length $m$.  Moreover, $\lambda_{n-1} + \lambda_{n} =
n^{2} + n(n-1) = 2n^{2}-n$ equals the number of principal avalanches
of size $\geq n-1$. But this number also equals the number
of $2$'s in every recurrent with a maximal $2$-string of size $\geq
n-1$.
%, since there are $2n$ recurrents with a maximal $2$-string of
%size $n-1$ and $\smax_{W_{n}}$ is the only recurrent with a maximal
%$2$-string of size $n$. 
Therefore, $\lambda_{1} + \dotsm + \lambda_{n}$
equals the number of $2$'s in every recurrent. 
Hence
\begin{align*}
\lambda_{0} &= n(l_{2n} - 2) - (\lambda_{1} + \dotsm + \lambda_{n}) 
= n(l_{2n} - 2) - 2n^{2} + n - \sum_{m=1}^{n-2} nmf_{2(n-m)} \\
&= n\left[l_{2n} - 2n - 1 - \sum_{m=1}^{n-2}mf_{2(n-m)}\right]
= n\left[l_{2n} - 2n - 1 - \sum_{m=2}^{n-1}(n-m)f_{2m}\right]\\
&= n\left[l_{2n} - n - 2 - \sum_{m=1}^{n-1}(n-m)f_{2m}\right]
= n\left[l_{2n} - n - 2 - \sum_{m=1}^{n-1}\sum_{k=1}^{m}
  f_{2k}\right]\\ 
&= n\left[l_{2n} - n - 2 - \sum_{m=1}^{n-1}\left(f_{2m+1}-1\right)\right] 
= n\left[l_{2n}  - 3 - \sum_{m=1}^{n-1}f_{2m+1}\right]\\ 
&= n(l_{2n}  - 2 - f_{2n}) = n(f_{2n+1}+f_{2n-1}-f_{2n}-2) 
= n(2f_{2n-1}  - 2) = 2n(f_{2n-1}-1). 
\end{align*}  
\end{proof}

In this case, $\lambda_{m}$ is also the number of
principal avalanches with burst size $m$ since every non-sink vertex
in $W_{n}$ is adjacent to the sink. Note also that as $n\to \infty$,
the proportion of avalanches of size $0$ is
\[\lim_{n\to \infty} \frac{2n(f_{2n-1}-1)}{n(l_{2n}-2)} = 1-\frac{1}{\sqrt{5}}. \]
Thus, recovering the last result in \cite[Section 2]{DaRoFPSAC03}.

%%%%%%%%%%%%%%%%%%%%%%%%%%%%%%%%%%%%%%%%%%%%%%%%%%%%%%%%%%%%%%%%%%%%%%%%%%%%%%%%

\section{Conclusions}\label{final}

In this paper, we introduce the \emph{multivariate avalanche
  polynomial} of a graph $G$. This new combinatorial object enumerates the
toppling sequences of all principal avalanches generated by adding a
grain of sand to any recurrent sandpile on $G$. We also give explicit
descriptions of the multivariate avalanche polynomials for trees,
cycles, complete, and wheel graphs. Furthermore, we show that certain
evaluations of this polynomial recover some important information. In
particular from this polynomial we can compute the distribution of the
size of all principal avalanches, that is, we recover the
(univariate) avalanche polynomial first introduced in
\cite{CDR04}. Moreover, a different evaluation gives rise  to the unnormalized distribution of burst sizes, that is, the number of
grains of sand that fall into the sink in a principal avalanche. The
burst size, introduced by Levine in \cite{L15}, is an important
statistic related to the relationship between the
threshold state of the 
fixed-energy sandpile and the stationary state of Dhar's abelian
sandpile. 
Of special interest is a description of the avalanche polynomial for
grids and the family of multiple wheel graphs introduced in
\cite{DaRoFPSAC03}.

%%%%%%%%%%%%%%%%%%%%%%%%%%%%%%%%%%%%%%%%%%%%%%%%%%%%%%%%%%%%%%%%%%%%%%%%%%%%%%%%

\section{Acknowledgements}

The authors would like to thank David Perkinson for suggesting to study the multivariate version of the avalanche polynomial and for his invaluable guidance and support throughout this project. We would also like to thank Andrew Fry, Lionel Levine, Christopher O'Neill, and Gautam Webb for their kind input and suggestions related to this work.

\bibliography{SandpileBib}{}
\bibliographystyle{plain}

\end{document}